\documentclass[10pt]{article}
\usepackage[margin=3cm]{geometry}
\usepackage{textcomp}
\usepackage[english]{babel} 
\usepackage[T1]{fontenc} 
\usepackage[utf8]{inputenc}
\usepackage{amsmath,amsthm,amssymb}
\usepackage{graphicx}
\usepackage{enumitem}

\newtheorem{theorem}{Theorem}[section]
\newtheorem{lemma}[theorem]{Lemma}
\newtheorem{corollary}[theorem]{Corollary}
\newtheorem{definition}[theorem]{Definition}

\newtheorem{hypothesis}[theorem]{Conjecture}

\newcommand{\distrib}[2]{
\left\langle #1,#2\right\rangle
}
\usepackage{url}

\title{3D Compton scattering imaging: study of the spectrum and contour reconstruction}

\author{Ga\"el Rigaud\\
Zentrum f\"ur Technomathematik, University of Bremen, Germany\\\url{gael.rigaud@math.uni-bremen.de}
}
\date{}

\begin{document}

\maketitle

\begin{abstract}
3D Compton scattering imaging is an upcoming concept exploiting the scattering of photons induced by the electronic structure of the object under study. The so-called Compton scattering rules the collision of particles with electrons and describes their energy loss after scattering. Although physically relevant, multiple-order scattering was so far not considered and therefore, only first-order scattering is generally assumed in the literature. The purpose of this work is to argument why and how a contour reconstruction of the electron density map from scattered measurement composed of first- and second-order scattering is possible (scattering of higher orders is here neglected). After the development of integral representations for the first- and second-order scattering, this is achieved by the study of the smoothness properties of associated Fourier integral operators (FIO). The second-order scattered radiation reveals itself to be structurally smoother than the radiation of first-order indicating that the contours of the electron density are essentially encoded within the first-order part. This opens the way to contour-based reconstruction techniques when using multiple scattered data. Our main results, modeling and reconstruction scheme, are successfully implemented on synthetic and Monte-Carlo data. 
\end{abstract}


\section{Introduction}

Computerized Tomography (CT) is a well-established and widely used technique which images an object by exploiting the properties of penetration of the x-rays. Due to the interactions of the photons with the atomic structure, the matter resists to the propagation of the photon beam, denoted by its intensity $I(\mathbf{x},\theta)$ at position $\mathbf{x}$ and in direction $\theta$, following the stationary transport equation
$$
\theta \cdot \nabla_\mathbf{x} I(\mathbf{x},\theta) + \mu_E(\mathbf{x}) I(\mathbf{x},\theta) = 0, \qquad \mathbf{x} \in \Omega \subset \mathbb{R}^3
$$
with $E$ the energy of the beam and $\cdot$ the standard inner product. The resistance of the matter is symbolized by the lineic attenuation coefficient $\mu_E(\mathbf{x})$. Solving this ordinary equation leads to the well-known Beer-Lambert law
$$
I(\mathbf{x}+ t\theta,\theta) = I(\mathbf{x},\theta) \exp\left(-\int_{0}^t \mu_E(\mathbf{x}+s\theta)ds\right) ,
$$
which describes the loss of intensities along the path $\mathbf{x}$ to $\mathbf{x}+t\theta$. The integral above corresponds to the so-called X-ray transforms, denoted here $\mathcal{R}$, which maps the attenuation map $\mu_E(x)$ into its line integrals, \textit{i.e.}
\begin{equation}\label{eq:Radontransform}
\left(\ln \frac{I(\mathbf{s},\theta)}{I(\mathbf{s}+T\theta,\theta)}=\right) \ g(\mathbf{s},\theta) = \mathcal{R} \mu_E (\mathbf{s},\theta) = \int_{0}^T \mu_E(\mathbf{s}+ t \theta) \mathrm{d} t
\end{equation}
with $\mathbf{s}$ the position of the source and $\mathbf{s}+T\theta$ the position of the detector. The task to recover $\mu_E$ from the data $g(\mathbf{s},\theta)$ can be achieved by various techniques such as the filtered back-projection (FBP) in two dimensions \cite{Natterer}.

Since the advent of CT, many imaging concepts have emerged and the need in imaging has grown. One can mention Single Photon Emission CT, Positron Emission Tomography or Cone-Beam CT for the standard system based on an ionising source. In these configurations, the energy has a very limited use but the idea of exploiting it in order to enhance the image quality, optimize the acquisition process or to compensate some limitations (such as limited angle issues) has led to various works \cite{H_73,AM_76,PRLYM_2009,Shefer2013,TM_2015,MLLF_2015,GG_2017,Fredenberg_2018}.

When one focuses on the physics between the matter and the photons, four types of interactions are observed: Thomson-Rayleigh scattering, photoelectric absorption, Compton scattering and pair production. In the classic range of applications of the x-rays or $\gamma$-rays \textit{i.e.} $[20,800]$ keV, the photoelectric absorption and the Compton scattering are the dominant phenomena which leads to a model for the lineic attenuation factor due to Stonestrom et al. \cite{Stonestrom} which writes
\begin{equation}\label{eq:stonestrom}
\mu_E(\mathbf{x}) = E^{-3} \lambda_{PE}(\mathbf{x}) + \sigma(E) n_e(\mathbf{x})
\end{equation} 
where $\lambda_{PE}$ is a factor depending on the materials and symbolizing the photoelectric absorption, $\sigma(E)$ the total-cross section of the Compton effect at energy $E$ and $n_e$ the electron density at $\mathbf{x}$. While the photoelectric absorption plays an important role in the attenuation of the photon beam, a measured photon either suffers no interaction (primary radiation) or is scattered (scattered radiation).

\begin{figure}[t]\centering
\includegraphics[width=0.5\linewidth]{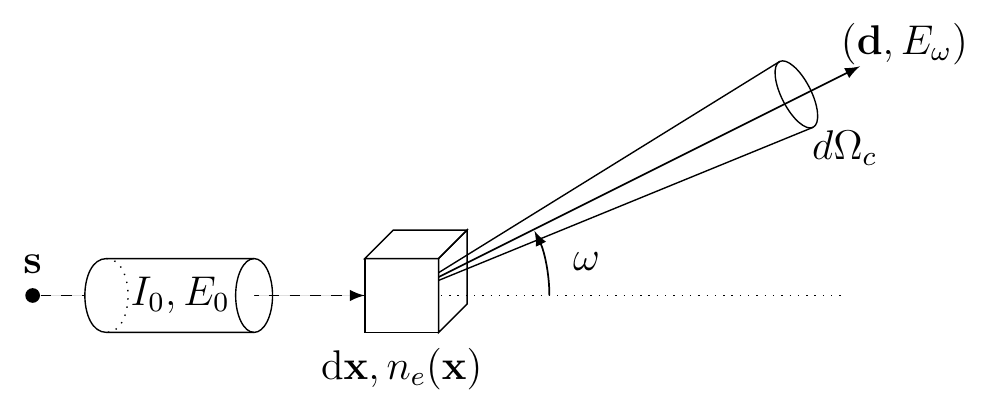}
\caption{Geometry of Compton scattering: the incident photon energy $E_0$ yields a part of its energy to an electron and is scattered with an angle $\omega$. \label{fig:Compton_fig}}
\end{figure}

The Compton effect stands for the collision of a photon with an electron. The photon transfers a part of its energy $E_0$ to the electron which suffers a recoil while the photon is scattered of an (scattering) angle $\omega$ with the axis of propagation, see Fig. \ref{fig:Compton_fig}. The energy of the photon after scattering is expressed by the Compton formula \cite{Compton_23},
\begin{equation}\label{eq:Compton_formula}
E_\omega = \frac{E_0}{1+\frac{E_0}{mc^2}(1-\cos\omega)},
\end{equation} 
where $mc^2 = 511$ keV represents the energy of an electron at rest.

The recent development of cameras able to measure accurately the energy of incoming photons opens the way to an innovative 3D imaging concept, Compton scattering imaging (abbreviated here by CSI). Although the technology of such detectors has not yet reached the same level of maturity as the one used in conventional imaging (CT, SPECT, PET), scientists have proposed and studied in the last decades different bi-dimensional systems, called Compton scattering tomography (CST), see e.g. \cite{Lale_59,Norton_94,Cesareo_02,ABE_11,AHD_90,BC_03,BCGLR_02,CV_73,EMBR_98,FC_71,GKAD_05,HH_2010,Hussein,Arendtsz,Guzzardi,Rigaud_SIIMS_17}. It is also possible to consider interior sources, for instance via the insertion of radiotracers like in SPECT. Then considering collimated detectors, it is possible to model the scattered flux through conical Radon transforms, see for instance \cite{Nguyen2011,Nguyen2017}. However, in this work, we consider only systems with external sources. 

In this paper we assume that the source is monochromatic, \textit{i.e.} it emits photons with same energy $E_0$. For sufficiently large $E_0$, larger than 80keV in medical applications, the Compton effect represents a substantial part of the radiation as more than 70\%  of the emitted radiation is scattered within the whole body. Therefore, the variation in terms of energy due to the Compton scattering, eq. (\ref{eq:Compton_formula}), will produce a polychromatic response to the monochromatic impulse of the source $\mathbf{s}$, see Figure \ref{fig:MultiScattering}. We decompose the spectrum $\mathrm{Spec}(E,\mathbf{d},\mathbf{s})$ measured at a detector $\mathbf{d}$ with energy $E$ as follows
\begin{equation}\label{eq:spectrum}
\mathrm{Spec}(E,\mathbf{d},\mathbf{s}) = \sum_{i=0}^\infty g_i(E,\mathbf{d},\mathbf{s}).
\end{equation}
In this equation, $g_0$ represents the primary radiation which crossed the object without being subject to the Compton effect. It corresponds to the signal measured in CT, eq. (\ref{eq:Radontransform}). The functions $g_i (\mathbf{d},E)$ corresponds to the photons that were measured at $\mathbf{d}$ with incoming energy $E$ after $i$ scattering events.

\begin{figure}[t]\centering
\includegraphics[width=\linewidth]{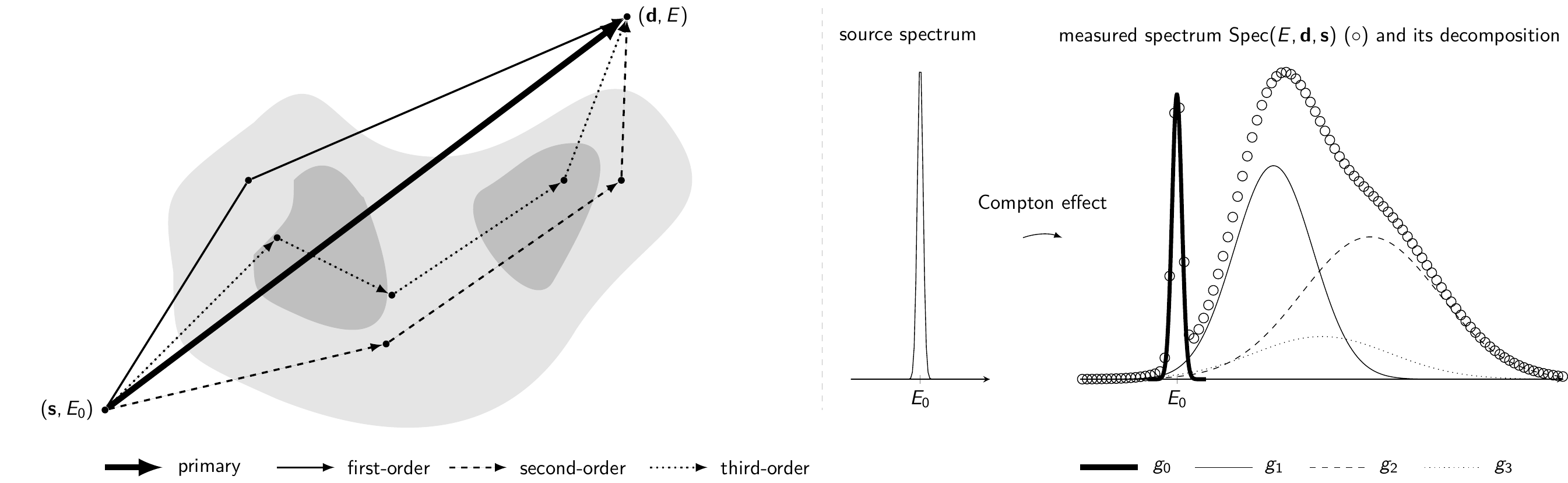}
\caption{Left: Illustration of the multiple scattering -- the detector $\mathbf{d}$ measures photons of energy $E$ that have not been scattered (primary) as well as scattered at different orders (here 1,2 and 3). Right: the spectrum of the detector (o) is illustrated via its decomposition in terms of type of scattered data and depicts the response to a monochromatic source due to the Compton effect. \label{fig:MultiScattering}}
\end{figure}

Until now the research was focused on solving the inverse problem $g_1 = \mathcal{T}(n_e)$ with $\mathcal{T}$ the operator modelling the measured first-order scattering. This approach will be very limited in practice as it neglects a substantial part of the measurement.  The purpose of this paper is to propose a strategy to extract features of the electron density $n_e$ from the spectrum in eq. (\ref{eq:spectrum}). Our main idea arises from studying the smoothness properties of the operators modeling the scattered radiations at various orders. It appears that the first scattering encodes the richest information about the contours while the scattering of higher orders lead to much smoother data. To make our point but also for the sake of implementation, we focus on the first and second scattering only, which means $g_1$ and $g_2$. The first- and second-order scattering represents 90-95\% of the scattered part of the spectrum, this is why $g_3,g_4, \ldots$ are treated here as noise. $g_0$ is also discarded as it brings no information related to the energy even though the primary radiation could be used as rich supplementary information in the future.

The manuscript is organized as follows. In Section \ref{sec:model}, we first recall the modeling for $g_1$ based on a weighted Radon transform along spindle tori given in \cite{RH_2018} and develop an integral representation for $g_2$ in Theorem \ref{theo:model_scat2}. We then validate both representations using a Monte-Carlo simulation for the acquisition process focusing on the first- and second-order scattering. In Section \ref{sec:smooth}, we focus on the study of the smoothness properties of $g_1$ and $g_2$. Both integral representations are difficult to handle due to the non-linearity with respect to the function of interest $n_e$ (electron density map of the object). In order to circumvent this difficulty, we propose to approximate them by linear operators, namely $\mathcal{L}_1$ and $\mathcal{L}_2$. Under some assumptions on the geometry of detectors, we could show they are Fourier integral operators (FIO) of order -1 and $-\frac74$ respectively with the corresponding smoothness properties, see Theorems \ref{theo:FIO_L1} and \ref{theo:FIO_L2}. Furthermore, Lemma \ref{lemma:immersion_sphere} delivers a condition on the detector set so that the corresponding measurement $g_1$ satisfies the semi-global Bolker condition. Under certain assumptions, we can describe the smoothness using Sobolev spaces, see Corollaries \ref{theo:smoothness_L1} and \ref{theo:smoothness_L2}. Since $g_2$ is shown structurally smoother than $g_1$, the singularities of $n_e$ are encoded essentially in $g_1$, this is why a contour-based reconstruction strategy is possible and proposed in Section \ref{sec:simu_MC}  following on from the results in \cite{RH_2018}. This approach has the advantage to get reconstructions of the contours without prior information about $g_2$ or even its modelling. The study of $g_2$ derived in this paper is thus important to understand how to deal with the full spectrum but its modelling does not need to be explicitely taken into account in practice. The global approach is validated by synthetical and Monte-Carlo simulations. 


\subsection{Recalls on Fourier integral operators}

The theory of Fourier integral operators (FIO) is a very powerful tool which allows for instance to characterize the smoothness properties of an FIO. For modalities involving Compton scattering, some cases were already studied in 2D \cite{WS_17} and 3D \cite{WQ_19} using microlocal analysis, however neglecting the weight functions and multiple scattering which are crucial in practice. 

For a vector $\mathbf{x} = (x_1,x_2,\ldots,x_n)$, we consider the following notations for the derivatives, 
$$
\partial_\mathbf{x} f = \sum_{i=1}^n \frac{\partial f}{\partial x_i} \mathrm{d} x_i, \qquad
\nabla_\mathbf{x} f = \left( \frac{\partial f}{\partial x_1}, \frac{\partial f}{\partial x_2}, \ldots, \frac{\partial f}{\partial x_n}    \right).
$$ 
We recall relevant results from \cite{Quinto} and from \cite{Hormander} focusing on the parametrization we consider in this work.
\begin{definition}[\cite{Quinto}]\label{def:FIO}
Let $\mathbb{R}\times \Theta \subset \mathbb{R}^m$, $\mathrm{dim}(\Theta) = m-1$ and $\Omega \subset \mathbb{R}^n$ be open subsets. A real valued function $\Phi \in C^\infty(\mathbb{R}\times \Theta \times \Omega \times \mathbb{R}\setminus\{0\})$ under the form
$$
\Phi(p,\theta,\mathbf{x},\sigma) = \sigma(p - \phi(\mathbf{x},\theta))
$$
with $\phi$ a given level-set function, is called a phase function if 
\begin{itemize}
\item $\Phi$ is positive homogeneous of degree 1 in $\sigma$, \emph{i.e.} $\Phi(p,\theta,\mathbf{x},r\sigma) = r \Phi(p,\theta,\mathbf{x},\sigma)$ for all $r >0$.
\item $(\partial_{p,\theta}\Phi, \partial_\sigma \Phi)$ and $(\partial_{\mathbf{x}} \Phi, \partial_\sigma \Phi)$ do not vanish for all $(p,\theta,\mathbf{x},\sigma) \in \mathbb{R} \times \Theta \times \Omega \times \mathbb{R}\setminus \{0\}$.
\item $\Phi$ is nondegenerate, \textit{i.e.} $\partial_{p,\theta,\mathbf{x}}(\partial_\sigma \Phi) \neq 0$ on the zero-set
$$
\Sigma_\Phi = \{(p,\theta,\mathbf{x},\sigma) \in \mathbb{R} \times \Theta \times \Omega \times \mathbb{R} \setminus \{0\} \ | \ \partial_\sigma \Phi = 0\}.
$$
\end{itemize}
An amplitude $a(\cdot) \in C^\infty(\mathbb{R} \times \Theta \times \Omega \times \mathbb{R})$ is of order $s$ if it satisfies: For every compact set $K\subset \mathbb{R} \times \Theta \times \Omega$ and for every index $\varrho,\rho$ and multi-index $\zeta,\gamma$, there is a constant $C$ such that
\begin{equation}\label{eq:cond_amplitude}
\left\vert \frac{\partial^{\varrho}}{\partial \sigma^\varrho} \frac{\partial^{|\zeta|}}{\partial \mathbf{x}^\zeta} \frac{\partial^{\rho}}{\partial p^\rho} \frac{\partial^{|\gamma|}}{\partial \theta^\gamma} a(p,\theta,\mathbf{x},\sigma) \right\vert \leq C (1+ |\sigma|)^{s-\varrho} \quad \forall \ (p,\theta,\mathbf{x}) \in K \ \text{and} \ \forall \ \sigma \in \mathbb{R}.
\end{equation}
The operator $\mathcal{F}$ defined as 
$$
\mathcal{F} u (y) = \int e^{i \Phi(p,\theta,\mathbf{x},\sigma)} a(p,\theta,\mathbf{x},\sigma) u(\mathbf{x}) \mathrm{d}\mathbf{x} \mathrm{d}\sigma
$$
is then called Fourier integral operator of order $s-\frac{n+m-2}{4}$ and we note $\mathcal{F}~\in~ I^{s-\frac{n+m-2}{4}}(\mathbb{R}\times\Theta,\Omega)$. Also, the canonical relation for $\mathcal{F}$ is defined by
$$
C_\mathcal{F} := \left\{ (p,\theta,\partial_{p,\theta} \Phi(p,\theta,\mathbf{x},\sigma), \mathbf{x},-\partial_\mathbf{x} \Phi(p,\theta,\mathbf{x},\sigma) \ |\ (p,\theta,\mathbf{x},\sigma) \in \Sigma_\Phi) \right\}.
$$ 
\end{definition}
We now provide some results of FIO in the Sobolev spaces.
\begin{definition}[Sobolev spaces]
Let $\Omega$ be an open subset of $\mathbb{R}^n$. The set $H_c^\nu(\Omega)$ is the set of all distributions $u$ with compact support in $\Omega$ such that the norm
$$
\Vert u \Vert_{\nu}^2 = \int_{\xi \in \mathbb{R}^n} |\hat{u}(\xi)|^2 (1+\Vert \xi \Vert^2)^\nu \mathrm{d}\xi, 
$$
with $\Vert \cdot\Vert$ the euclidean norm, is finite with $\hat{u}$ the $n$-dimensional Fourier transform. The set $H_{loc}^\nu(\Omega)$ is the set of all distributions $u$ supported in $\Omega$ 
such that $\chi u \in H_c^\nu(\Omega)$ for all $C^\infty$-smooth functions $\chi$ with compact support in $\Omega$.
\end{definition}

\begin{theorem}[\cite{Hormander}] \label{theo:Hormander}
Let $\mathcal{F}$ be a FIO of order $k$ and assume the natural projection $\Pi_{L} : C_\mathcal{F} \to \mathbb{R}\times \Theta$ is an immersion, \textit{i.e.} with injective derivative. Then, 
$$ 
\mathcal{F} \ : \ H_c^\nu(\Omega) \rightarrow H_{loc}^{\nu-k}(\mathbb{R}\times \Theta)
$$
is continuous.
\end{theorem}

\begin{theorem}[\cite{Hormander}] \label{theo:Hormander2}
Let be $\mathcal{F}$ a FIO with phase $\Phi$. The natural projection $\Pi_{L}~:~C_\mathcal{F} \to \mathbb{R}\times \Theta$ is an immersion if 
$$
\mathrm{det}(\nabla_x \phi, \partial_{\theta_{1}} \nabla_x \phi,\ldots, \partial_{\theta_{n-1}} \nabla_x \phi) \neq 0 \quad \forall (x,\theta) \in \Omega \times \Theta.
$$
\end{theorem}

We also need the following property from integral geometry throughout the article. 
\begin{theorem}[\cite{Palamodov}]\label{theo:integration_dirac}
Denoting by $\delta$ the Dirac delta distribution, it holds
$$
\int_\Omega \Vert\nabla_\mathbf{x} \gamma(\mathbf{x})\Vert \delta(\gamma(\mathbf{x})) f(\mathbf{x}) \mathrm{d}\mathbf{x} = \int_{\gamma^{-1}(0)} f(\mathbf{x}) \mathrm{d}\pi(\mathbf{x}) 
$$
with $\gamma:\mathbb{R}^n \to \mathbb{R}$ a continuously differentiable function with non-vanishing gradient, $\gamma^{-1}(0)$ the $(n-1)$-dimensional surface defined by $\gamma(\mathbf{x}) = 0$ and $\pi(\mathbf{x})$ the corresponding measure.
\end{theorem}

\section{On modeling the spectrum}\label{sec:model}

In this section, we provide integral representations for the spectrum $\mathrm{Spec}(\mathbf{d},E)$. As said above, we focus on $g_1$ and $g_2$ while we treat $g_3, g_4, \ldots$ as noise $\eta$, \emph{i.e.}
$$
\mathrm{Spec}(E,\mathbf{d},\mathbf{s}) = (g_1 + g_2 + \eta)(E,\mathbf{d},\mathbf{s}).
$$
The first-order scattered radiation $g_1$ was already studied in the literature, see \cite{RH_2018}. The next subsection summarizes the modeling of $g_1$ by a toric Radon transform.

\subsection{First-order scattering and toric Radon transforms}


We consider a photon beam traveling from $\mathbf{m}$ to $\mathbf{n}$. The interactions of the photon beam with the matter leads its intensity to reduce exponentionally due to the resistance of the matter to the propagation of light, the attenuation factor or Beer-Lambert law, and proportionally to the square of the travelled distance, the photometric dispersion. This latter comes from the propagation of photons not on straight lines but within solid angles with vertex being the emission or scattering point, see also Figure \ref{fig:Compton_fig}.
To represent these two physical factors, we define the following mapping of the attenuation map $\mu_E$ 
$$
A_{E}(\textbf{m} , \textbf{n}) := \Vert \mathbf{n}-\mathbf{m} \Vert^{-2} \exp\left( - \Vert \mathbf{n}-\mathbf{m} \Vert \int_0^1  \mu_E \left(\mathbf{m} + t (\mathbf{n}-\mathbf{m})\right) \mathrm{d}t \right).
$$
Using this notation and considering an ionising source $\mathbf{s}$ with energy $E_0$ and photon beam intensity $I_0$, the variation of the number of photons $N_c$ scattered at $\mathbf{x}$ and detected at $\mathbf{d}$ with energy $E_\omega$, see \cite{Driol}, can be expressed as
\begin{equation}\label{eq:nb_photon_Compton}
 \mathrm{d} N_c(\mathbf{x}, \mathbf{d}, \mathbf{s}) = \frac{I_0 r_e^2}{4} P(\omega) A_{E_0}(\mathbf{s} , \mathbf{x}) A_{E_\omega}(\mathbf{x}, \mathbf{d}) n_e(\mathbf{x}) \mathrm{d}\mathbf{x},
\end{equation} 
where $P(\omega)$ stands for the Klein-Nishina probability \cite{KN}, and $r_e$ is the classical radius of an electron. This formula describes the evolution of the first scattered radiation which is detected at a given energy and at a given detector position.

\begin{figure}[t]\centering
\includegraphics[width=0.5\linewidth]{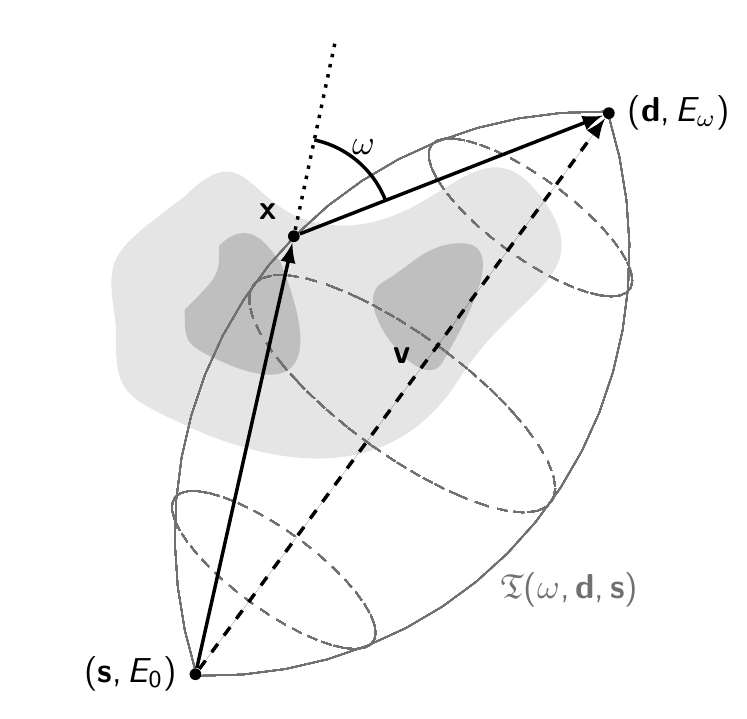}
\caption{Geometric representation of the torus $\mathfrak{T}(\omega,\mathbf{s},\mathbf{d})$ oriented by $\mathbf{v} = \mathbf{d}-\mathbf{s}$ and scaled by the scattering angle $\omega$ with $\sphericalangle (\mathbf{x}-\mathbf{s},\mathbf{d} - \mathbf{x}) = \pi - \omega$ for all $\mathbf{x} \in \mathfrak{T}(\omega,\mathbf{s},\mathbf{d})$. The plain arrows symbolize illustrations of the flight of a measured scattered photon. \label{fig:Geometry_CSI_1}}
\end{figure}

Due to the Compton formula eq. (\ref{eq:Compton_formula}), the scattered energy $E_\omega$ corresponds to a unique scattering angle $\omega$ and thus delivers a specific geometry when focusing on the first scattering. Indeed, all scattering points $\mathbf{x}$ responsible for a detected scattered photon at energy $E_\omega$ belong to
$$
\mathfrak{T}(\omega,\mathbf{d},\mathbf{s}) = \left\{ \mathbf{x} \in \mathbb{R}^3 \ : \ \sphericalangle (\mathbf{x}-\mathbf{s},\mathbf{d} - \mathbf{x}) = \pi - \omega \right\}
$$
with $\sphericalangle (\cdot,\cdot)$ the angle between two vectors. As depicted by Figure \ref{fig:Geometry_CSI_1}, this set corresponds to a part of a spindle torus, see \cite{Wolfram}. More precisely, $\mathfrak{T}(\omega,\mathbf{d},\mathbf{s})$ denotes the \textit{lemon} part of the spindle torus for $\omega \in [0,\frac{\pi}{2}]$ and to its \textit{apple} part for $\omega \in [\frac{\pi}{2},\pi]$. Assuming now $\mathbf{d}$ to be a \textit{point} detector and integrating over $\Omega := \mathrm{dom}(n_e)$ the equation (\ref{eq:nb_photon_Compton}), one obtains an integral representation of the detected first scattered radiation, \textit{i.e.} 
$$
g_1(E_\omega,\mathbf{d},\mathbf{s}) = N_c(E_\omega,\mathbf{d},\mathbf{s}) \propto \int_{\mathbf{x}\in\mathfrak{T}(\omega,\mathbf{d},\mathbf{s})}  A_{E_0}(\textbf{s} , \textbf{x}) A_{E_\omega}(\textbf{x} , \textbf{d}) n_e(\textbf{x}) \mathrm{d}\mathbf{x}.
$$

As proven in \cite{RH_2018}, the quantities $N_c(\mathbf{s},\mathbf{d},E_\omega)$ can be interpreted as a weighted toric Radon transform. Using the properties of integration of the dirac distribution given in Theorem \ref{theo:integration_dirac}, it follows that the first scattering radiation, $g_1(E_\omega,\mathbf{d},\mathbf{s})$, can be modelled after suited change of variables by
\begin{equation}\label{eq:FIO_g1}
\mathcal{T}(n_e)(p,\mathbf{d},\mathbf{s}) := \int_\Omega \Vert\nabla_\mathbf{x} \phi(\mathbf{x},\mathbf{d},\mathbf{s})\Vert \; w_1(n_e)(\mathbf{x},\mathbf{d},\mathbf{s}) \; n_e(\mathbf{x}) \; \delta(p - \phi(\mathbf{x},\mathbf{d},\mathbf{s})) \; \mathrm{d} \mathbf{x}
\end{equation}
with $w_1(n_e) = A_{E_0} \cdot A_{E_\omega}$, $p = \cot \omega$, where
$\cot : (0,\pi) \mapsto \mathbb{R} $
and the characteristic function given by
\begin{equation}\label{eq:phase_phi}
\phi(\mathbf{x},\mathbf{d},\mathbf{s})  = \frac{\kappa(\mathbf{x},\mathbf{d},\mathbf{s}) - \rho(\mathbf{x},\mathbf{d},\mathbf{s})}{\sqrt{1-\kappa^2(\mathbf{x},\mathbf{d},\mathbf{s})}}
\end{equation}
where
\begin{equation}\label{eq:def_kappa_rho}
\kappa(\mathbf{x},\mathbf{d},\mathbf{s}) = \frac{(\mathbf{x}-\mathbf{s})}{\Vert \mathbf{x}-\mathbf{s} \Vert } \cdot \frac{(\mathbf{d}-\mathbf{s})}{\Vert \mathbf{d}-\mathbf{s} \Vert } 
\quad \text{and} \quad 
\rho(\mathbf{x},\mathbf{d},\mathbf{s}) = \frac{\Vert \mathbf{x}-\mathbf{s} \Vert}{\Vert \mathbf{d}-\mathbf{s} \Vert}.
\end{equation}
In the following, we will write $\kappa(\mathbf{x})$ and $\rho(\mathbf{x})$ instead for the sake of simplicity. \\[1em]

\textbf{Remark:} Depending on the resolution of the detector, the spectrum can be parametrized either by the energy $E_\omega$, the scattering angle $\omega$ or $p=\cot \omega$. These three parameters are analytically equivalent on the considered range of energy, $(E_0,E_\pi)$, as they are related by diffeomorphisms with non-vanishing derivatives. For the study of the operator $\mathcal{T}$ we chose the parameter $p$ over the energy or scattering angle without loss of generality and for the sake of a simpler phase $\phi$.\\

\subsection{Modeling the second order scattered radiation}

We derive now an integral representation of the second-order scattered radiation, noted $g_2$. The higher-orders are difficult to handle but we expect a similar approach for the orders larger than 2, see Conjecture \ref{hypo}.

The derivation of the second scattering will use the modelling of the first scattering seen above. In this case, a measured photon is scattered twice instead of once before being detected. We denote by $\mathbf{x}$ and $\mathbf{y}$ the first and second scattering points respectively and by $\omega_1$ and $\omega_2$ the first and second scattering angles respectively. The keyidea is to interpret each first scattering point $\mathbf{x}$ as a new polychromatic \textit{source} regarding the following second scattering points. 

Let us consider a detector $\mathbf{d}$ and a detected energy $E$. Due to the Compton formula, the scattering angles must satisfy 
\begin{equation}\label{eq:Relation_cos_Energy}
\cos \omega_1 + \cos \omega_2 = 2 - mc^2 \left( \frac{1}{E} - \frac{1}{E_0} \right) =: \lambda(E) \in (0,2).
\end{equation}
The boundaries 0 and 2 are here excluded since they correspond to the degenerated cases - primary ($\omega_1 = \omega_2 = 0$) and backscattered ($\omega_1 = \pi$ and $\omega_2 = 0$) radiation - of the torus. In consequence, the second scattering angle can be expressed as a function of the first one:
$$
\omega_2(\omega_1) = \arccos\left(  \lambda(E) -\cos \omega_1 \right).
$$
The first angle $\omega_1$ means that a photon arriving at $\mathbf{x}$ with energy $E_0$ is scattered by an angle $\omega_1$ and has afterwards an energy $E_{\omega_1}$. Such photons belong then to the cone of aperture $\omega_1$, vertex $\mathbf{x}$ and direction $\mathbf{x} - \mathbf{s}$, i.e. to the cone
$$
\mathfrak{C}(\omega_1,\mathbf{x},\mathbf{s}) := \left\{ \mathbf{y} \in \mathbb{R}^3 \ : \ \psi(\mathbf{y},\mathbf{x},\mathbf{s}) = \cos \omega_1  \right\}
$$
with characteristic function
\begin{equation}\label{eq:phase_cone}
\psi(\mathbf{y},\mathbf{x},\mathbf{s}) = \frac{\mathbf{y}-\mathbf{x}}{\Vert \mathbf{y}-\mathbf{x} \Vert} \cdot \frac{\mathbf{x}-\mathbf{s}}{\Vert \mathbf{x}-\mathbf{s}\Vert}.
\end{equation}
$\psi$ leads to a standard representation of the cone, \textit{i.e.}
$$
\mathfrak{C}(\omega_1,\mathbf{x},\mathbf{s}) = \left\{ \mathbf{x} + \frac{t}{\cos \omega_1} R_{1} 
\left(
\begin{array}{c}
\sin \omega_1 \cos \varphi \\ 
\sin \omega_1 \sin \varphi \\ 
\cos \omega_1 
\end{array}\right) 
\ : \ t \in \mathbb{R}^+, \varphi \in [0,2\pi) \right\}
$$
with $R_1$ the rotation matrix which maps 
$
\frac{\mathbf{d} - \mathbf{x}}{\Vert \mathbf{d} - \mathbf{x} \Vert} 
$
onto
$
\frac{\mathbf{x} - \mathbf{s}}{\Vert \mathbf{x} - \mathbf{s} \Vert}
$. Such a matrix can be computed using the Rodrigues formula.

\begin{figure}[t]\centering
\includegraphics[width=0.75\linewidth]{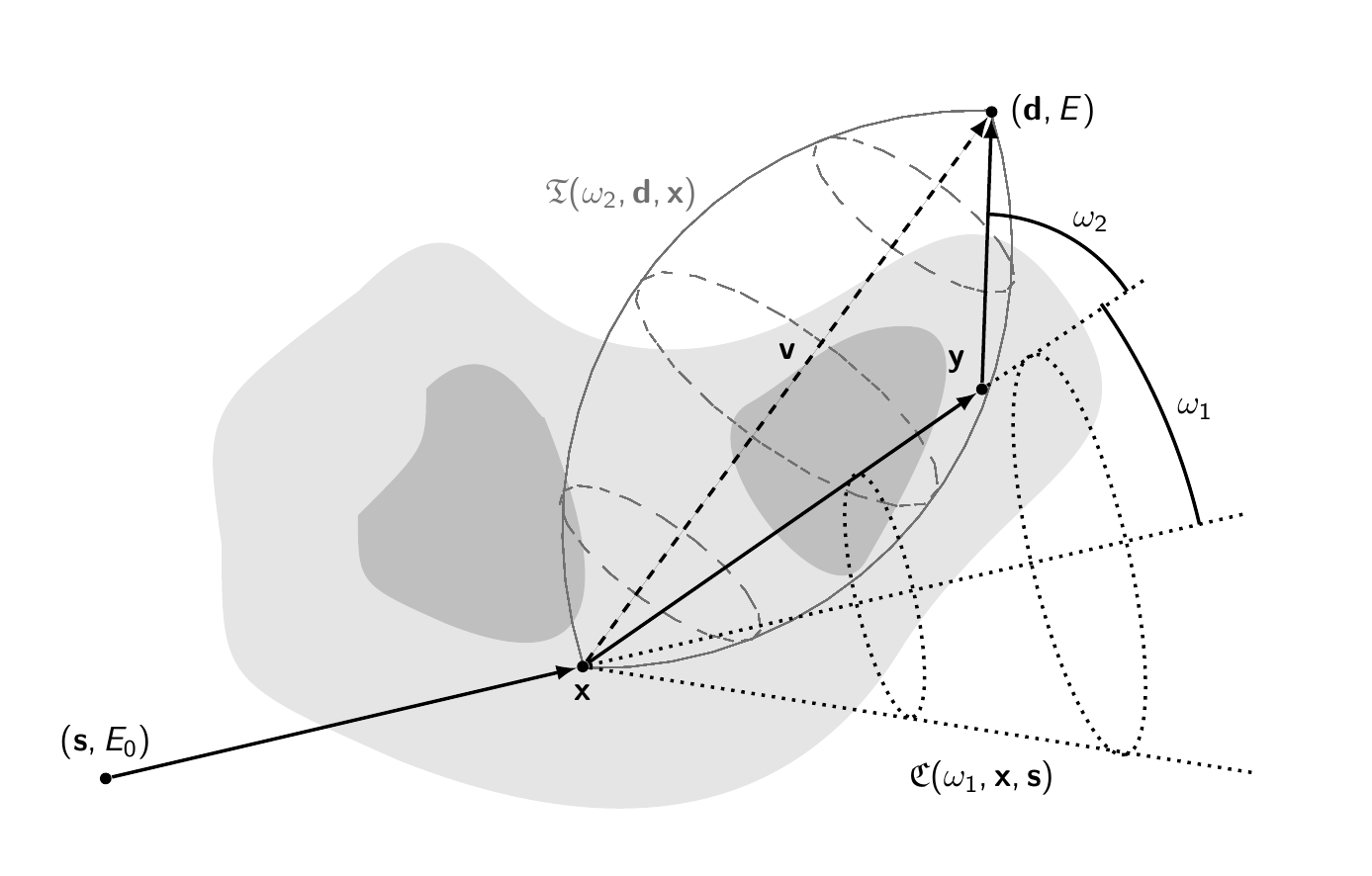}
\caption{Geometry of the second-order scattering: a first scattering point M becomes a new source regarding the next scattering event and emits photons with energy $E_{\omega_1}$ along corresponding cones $\mathfrak{C}(\omega_1,\mathbf{x})$. Beside the next scattering event occurs with an angle $\omega_2$ leading to a scattering site N belonging to the corresponding spindle torus $\mathfrak{T}(\omega_2,\mathbf{x},\mathbf{d})$. Thence, N belongs to the intersection between the cone and  the spindle torus. \label{fig:Geometry_CSI_2}}
\end{figure}

The second angle $w_2$ means that a photon arriving at $\mathbf{y}$ with energy $E_{\omega_1}$ is scattered by an angle $\omega_2$ and is afterwards detected at $\mathbf{d}$ with an energy $E$. As seen in the previous section, such photons belong to the torus with fixed points $\mathbf{x}$ and $\mathbf{d}$, $\mathfrak{T}(\omega_2,\mathbf{x},\mathbf{d})$, such that
$$
\cot \omega_2(\omega_1) = \phi(\mathbf{y},\mathbf{d},\mathbf{x})
$$
with $\phi$ defined in eq. (\ref{eq:phase_phi}). A parametric representation of the spindle torus is given in \cite{RH_2018} and writes
\begin{equation}\label{eq:def_torus_scat1}
\mathfrak{T}(\omega_2,\mathbf{d},\mathbf{x}) = 
\left\{
\mathbf{x} + \Vert \mathbf{d}-\mathbf{x}\Vert \frac{\sin(\omega_2 - \alpha)}{\sin \omega_2} R_2
\left(
\hspace{-4pt}
\begin{array}{c}
\sin \alpha \cos \beta \\ 
\sin \alpha \sin \beta \\ 
\cos \alpha 
\end{array}
\hspace{-4pt}
\right) 
: \alpha \in [0,\omega_2], \beta \in [0,2\pi)
\right\}
\end{equation}
with $R_2$ the rotation matrix which maps 
$
e_z = (0,0,1)^T
$
onto
$
\frac{\mathbf{d} - \mathbf{x}}{\Vert \mathbf{d} - \mathbf{x} \Vert}
$.  Regarding the vertex $\mathbf{x}$, the torus corresponds to a simple unit vector multiplied by the radius
$$
r(\omega_2,\alpha) = \Vert \mathbf{d}-\mathbf{x}\Vert \left( \cos \alpha - \frac{\sqrt{1-\cos^2 \alpha}}{\tan \omega_2} \right).
$$

Since the new \textit{source} $\mathbf{x}$ emits photons with various energies in the corresponding cone, the Compton formula and the relationship $\omega_2(\omega_1)$ implies that a photon detected at $\mathbf{d}$ with energy $E$ and scattered at $\mathbf{x}$ with angle $\omega_1$ must belong to the intersection
$$
\mathbf{y} \in \mathfrak{T}(\omega_2(\omega_1),\mathbf{d},\mathbf{x}) \cap \mathfrak{C}(\omega_1,\mathbf{x},\mathbf{s}). 
$$
This intersection and the principle described above is depicted in Figure \ref{fig:Geometry_CSI_2}. 

We want to provide an explicit characterization of the intersection. Therefore, to simplify the analysis, we consider the torus to be oriented in the direction $e_z$, which is achieved by applying the matrix $R_2^{-1}$. In this setting, $\cos \alpha$ corresponds to the third component of the normalized vector. Since we are interested in the intersection of the cone and of the torus, one gets
$$
\cos \alpha = R_1(3,1) \sin \omega_1 \cos \varphi
+ R_1(3,2) \sin \omega_1 \sin \varphi + R_1(3,3) \cos \omega_1 =: z_\cap
$$ 
with $R_1(3,\cdot)$ the third row of the rotation matrix $R_1$. Using the parametrisation of the cone, one gets for the intersection, the following radius 
\begin{equation}\label{eq:radius_intersection}
r_\cap := \Vert \mathbf{d}-\mathbf{x}\Vert \left( z_\cap - \frac{\sqrt{1-z_\cap^2}}{\tan \omega_2} \right)
\end{equation}
and thus
\begin{equation}\label{eq:coordinates_intersection}
\mathbf{y}_\cap = \mathbf{x} + r_\cap \ R_2 R_1 \left(
\begin{array}{c}
\sin \omega_1 \cos \varphi \\ 
\sin \omega_1 \sin \varphi \\ 
\cos \omega_1 
\end{array}\right)  \qquad \text{if } r_\cap > 0.
\end{equation}
Since the torus is oriented ($\mathbf{x}$ to $\mathbf{d}$), we can discard the intersection between the cone and the opposite torus which corresponds to negative radii in order to fit the physics. In practice, we can ignore them since they lead to detected energies outside the considered range. Also, the case $r_\cap = 0$ is discarded as it would correspond to two successive scattering occuring at the exact same location. We have now the tools to give an integral representation of $g_2$.
\begin{theorem}\label{theo:model_scat2}
Considering an electron density function $n_e(\mathbf{x})$ with compact support $\Omega$, a monochromatic source $\mathbf{s}$ with energy $E_0$ as well as a detector $\mathbf{d}$ both located outside $\Omega$. Then, the number of detected photons scattered twice arriving with an energy $E$ are given by
\begin{equation}\label{eq:FIO_g2}
 g_2(E,\mathbf{d},\mathbf{s}) = \mathcal{S}(n_e)(E,\mathbf{d},\mathbf{s})  = \int_\Omega \int_0^{2\pi} \hspace{-5pt} \int_0^\pi
w_2(\mathbf{y}_\cap,\mathbf{x}; \omega_1,\varphi, \mathbf{d},\mathbf{s})
n_e(\mathbf{x}) n_e(y_\cap) \mathrm{d}S_\cap(\omega_1,\varphi) \mathrm{d}\mathbf{x}
\end{equation}
with the physical factors symbolized by
$$
w_2(\mathbf{y}_\cap,\mathbf{x}; \omega_1,\varphi, \mathbf{d},\mathbf{s}) = 
 A_{E_0}(\mathbf{s},\mathbf{x}) A_{E_{\omega_1}}(\mathbf{x},\mathbf{y}_\cap) A_{E_{\omega_2}}(\mathbf{y}_\cap,\mathbf{d})
$$
and the differential form of the intersection given by
$$
\mathrm{d}S_\cap(\omega_1,\varphi) = r_\cap  \sqrt{\sin^2\omega_1 (r_\cap^2 +  \left( \partial_{\omega_1} r_\cap  \right)^2) + \left( \partial_{\varphi} r_\cap \right)^2} \  \mathrm{d}\omega_1 \mathrm{d}\varphi
$$
in which $\mathbf{y}_\cap$ and $r_\cap$ are described in eqs (\ref{eq:coordinates_intersection}) and (\ref{eq:radius_intersection}).
\end{theorem}
\begin{proof}
The derivation starting from an extension of eq. (\ref{eq:nb_photon_Compton}) along with the computation of the differential element along the intersection of the torus and the cone is given in the Appendix.
\end{proof}

This Theorem delivers an integral representation for the second scattering in any architecture for the displacement of the source and detectors.

In order to express the second-order scattered data $g_2$ using a Dirac delta function, one needs to find the corresponding characteristic function of the intersection between the cone and the spindle torus. The latter is delivered by our condition on the scattering angles and the energy derived from the Compton formula eq. (\ref{eq:Relation_cos_Energy}). Considering the characteristic functions from eqs. (\ref{eq:phase_phi}) and (\ref{eq:phase_cone}), eq. (\ref{eq:Relation_cos_Energy}) yields 
\begin{equation}\label{eq:phase_g2}
\psi(\mathbf{y},\mathbf{x},\mathbf{s}) + \cos \left( \cot^{-1} \phi(\mathbf{y},\mathbf{x},\mathbf{d})\right) = \lambda(E) = \Psi(\mathbf{y},\mathbf{x},\mathbf{d},\mathbf{s}) 
\end{equation}
with $\cot^{-1} : \mathbb{R} \mapsto (0,\pi)$. Finally, using Theorem \ref{theo:integration_dirac} and the appropriate change of variable $\lambda$ instead of $E$, $g_2$, defined in Theorem \ref{theo:model_scat2}, yields
\begin{equation}\label{eq:def_S}
\mathcal{S}(n_e)(\lambda,\mathbf{d},\mathbf{s}) = \int_\Omega \int_{\Omega\setminus\{\mathbf{x}\}} \mathcal{W}_2(n_e)(\mathbf{x},\mathbf{y},\mathbf{d},\mathbf{s})
n_e(\mathbf{x}) n_e(\mathbf{y}) \delta(\lambda - \Psi(\mathbf{y},\mathbf{x},\mathbf{d},\mathbf{s})) \mathrm{d}\mathbf{y} \mathrm{d}\mathbf{x}
\end{equation}
with $\mathcal{W}_2(n_e) = \Vert \nabla_{\mathbf{x},\mathbf{y}} \Psi\Vert w_2(n_e)$.  The case $\mathbf{x} = \mathbf{y}$ is discarded as a source cannot be a new scattering point (it was the case $\mathbf{s} = \mathbf{x}$ for the first scattering).

The integral representations for $g_1$ and $g_2$ are now validated and compared to Monte-Carlo simulations.

\subsection{Comparisons with Monte-Carlo simulations}\label{ssec:MCPSF}

\begin{figure}[t]\centering
\includegraphics[width = 0.5\linewidth]{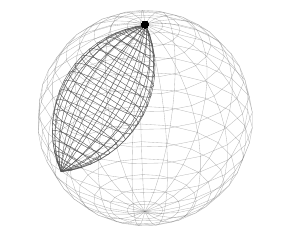}
\caption{Scanning geometry considered here: the source is fixed and a set of detectors are located on a sphere. The position of the source (here north pole) is not important as long as it remains outside the object.\label{fig:config_CSI}}
\end{figure}

The modellings of $g_1$ and $g_2$ have the advantage to be suited for any displacement/position of the detectors, represented by $\mathbb{D}$. We can then apply the approach for various architectures of 3D CSI. For instance, with the source and one detector drawing diameters of a given sphere \cite{WL_17}, or fixing the source and considering many detectors along a cylinder or a sphere. Here we chose the latter case which provides a modality rotation-free and no limited angle issues for the measurement, see Figure \ref{fig:config_CSI}. 

A cross-section of the \textit{scanner} is depicted in Figure \ref{fig:psf_MCvsAna}. The monochromatic source, emitting at $E_0 = 662$ keV, is fixed and located under the object while the detectors are located on a sphere (the half-circle on the slice) of diameter equal to 40 cm. Each detector is a disk of 2mm of radius. The volume to reconstruct is represented by a cube in the middle of 10x10x10cm$^3$. We consider that the source has emitted 10$^{11}$ photons; a sufficient amount in order to limitate the Poisson noise in the Monte-Carlo data. The electron density map is scaled on the value of water, \textit{i.e.} 3.23 $\cdot 10^{23}$ electrons per cm$^{-3}$, noted $n_w$. 

\begin{figure}[!t]\centering
\includegraphics[width=\linewidth]{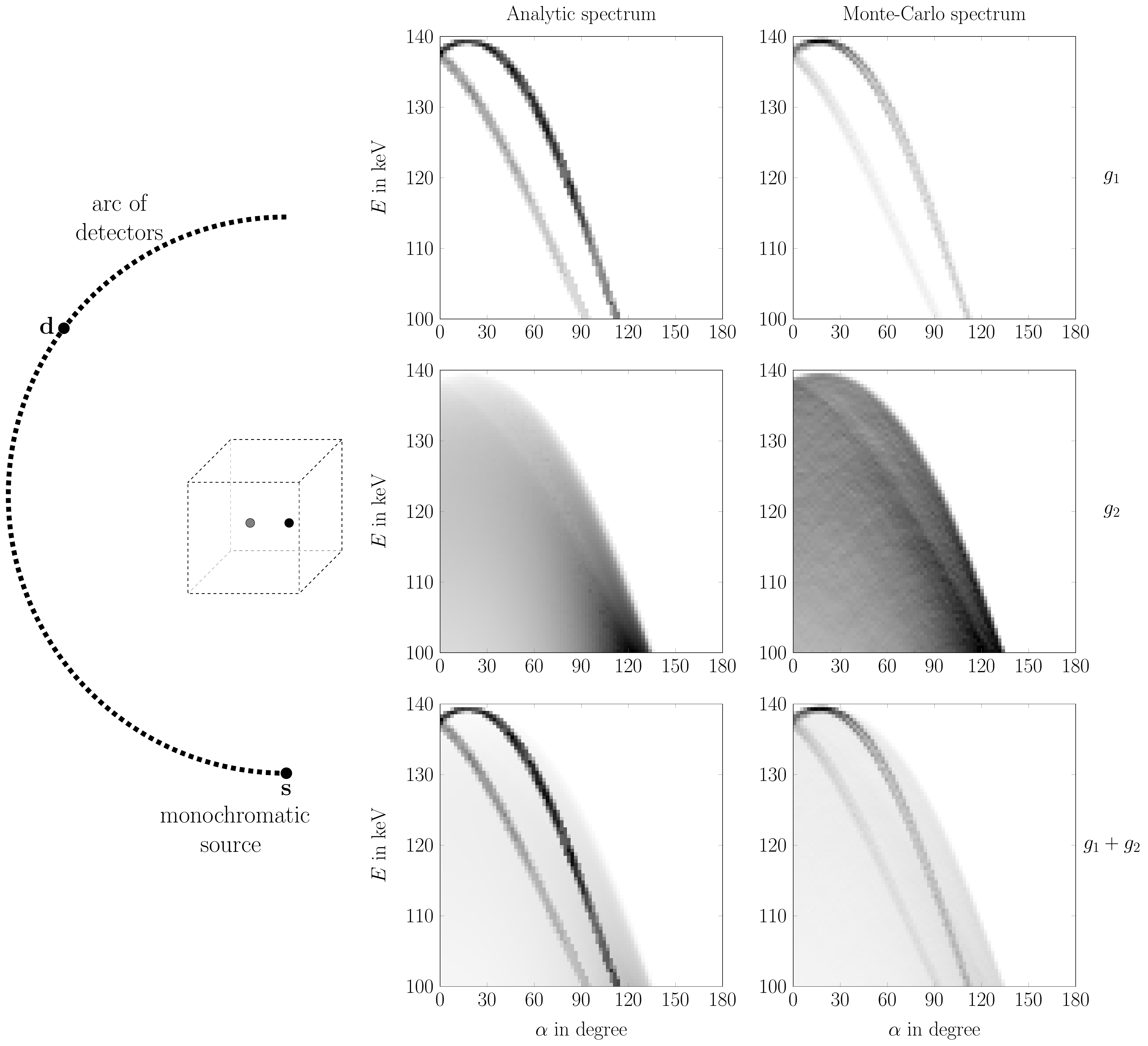}
\caption{Illustration of the \textit{point spread function} for the first-order scattered radiation $g_1$ and the second-order scattering radiation $g_2$. The sketch on the left depicts the configuration: a monochromatic source is fixed and the detectors are located on a half-circle. Our volume (the cube in the center) consists of two points with intensities $n_w$ and $2 n_w$. The figures in the middle correspond to an implentation of the integral representation of the data while the ones on the right gives the corresponding Monte-Carlo data. \label{fig:psf_MCvsAna}}
\end{figure}

We compare our model of the spectrum for first- and second-order scattering with ground truth data obtained by Monte-Carlo simulation. To view the response of the different operators, we consider the well-known \textit{point spread function}. Since we want to validate the modeling of the second scattering, we have to consider two \textit{points} (small spheres each of radius 2mm). Figure \ref{fig:psf_MCvsAna} displays the different results for one arc of detectors (the complete dataset is obtained for detectors along the complete sphere). Analytic data (middle column) are compared with ground truth data (right column). The first row shows the part of the first-scattering within the spectrum, $g_1$, the second row the part of the second scattering $g_2$, and the third row the spectrum $g_1 + g_2$ (neglecting higher-order scattering). The analytic spectrum and the different parts match with the realistic spectrum, in particular the singularitues. One can observe some variations in the intensities which arise from the discretization of the modeling. In particular, for the Monte-Carlo simulations, the detector is no more a point but a disk of size 4mm and the detection of the energy is made on a range of energy $E\pm \Delta E$ and not a single value $E$. Here $\Delta E$ is set 0.25 keV. Combined together, it means we do not integrate over spindle torus (or over the intersection with the cone for $g_2$) but over a strip around the spindle torus. Also, the computation of $g_2$ as an integral over manifolds of dimension 5 (the intersection between torus and cylinder for each potential first-scattering points) is numerically difficult to handle. However, the analytical modeling is revealed to appropriately model the spectrum as shown by the reconstructions in Section \ref{sec:simu_MC}. 

Since the model for the second-order scattering, $\mathcal{S}$, makes intervene a second integral over $\Omega$ (indeed all points of $\Omega$ are sources for the second scattering event), it is expected that $g_2$ should be smoother than $g_1$. In the next section, we justify and quantify this statement using the theory of FIO. Furthermore, the obtained results  on the forward operators show that the modeling of $g_2$ does not matter in practice since the extraction of the contours developed in Section \ref{sec:simu_MC} focuses on the singularities carried essentially by $g_1$ and by the same way circumvent the impact of the rest of the spectrum such as $g_2$. The explicit representation for the operator $\mathcal{S}$ is thus never used by the algorithm.

\section{Smoothness properties of the spectrum}\label{sec:smooth}

While it is difficult in practice to describe accurately physical variables as continuous functions, they are usually assumed to be piecewise continuous $L_2$ functions. However, (ii) requires the electron density $n_e$, more generally the attenuation map $\mu_E$, to be a $C^\infty(\Omega)$ function. On the other hand, solving ill-posed inverse problems of type $\mathcal{A} u = v$ with operator $\mathcal{A}:L_2(\Omega) \to L_2(\mathbb{R}\times \Theta)$  requires a regularization step that can be seen, at least for linear problems, as finding a reconstruction kernel $k_\gamma$ such that
$$
u_\gamma = \distrib{u}{e_\gamma} = \distrib{v}{k_\gamma}
$$ 
with $e_\gamma$ a suited mollifier, see \cite{Louis}. If $e_\gamma$ is a $C^\infty$-mollifier (for instance the Gaussian function), then the set of regularized solutions $u_\gamma$ will be $C^\infty$ functions which approximate the exact solution $u$. Built on this observation, our strategy is to interpret the spectral data $\mathcal{T}(n_e)$ and $\mathcal{S}(n_e)$ as pointwise approximations of Fourier integral operators which can be studied. These FIO are constructed from $\mathcal{T}$ and $\mathcal{S}$ by splitting 
the dependency on $n_e$ between the weight functions $w_1,w_2$ and the integrand itself which is achieved by considering a function $f\in L_2(\Omega)$ such that 
$
n_e = \distrib{f}{e_\gamma} 
$
with  $e_\gamma$ an appropriate $C^\infty$-mollifier. \\ 

Defining by 
\begin{equation}\label{eq:def_L1}
\mathcal{L}_1(f,n_e)(p,\mathbf{d},\mathbf{s}) = \int_{\Omega} \mathcal{W}_1(n_e) (\mathbf{x},\mathbf{d},\mathbf{s}) \ f(\mathbf{x}) \ \delta(p-\phi(\mathbf{x},\mathbf{d},\mathbf{s})) \ \mathrm{d}\mathbf{x}
\end{equation}
with $\mathcal{W}_1(n_e) = \Vert \nabla_\mathbf{x} \phi\Vert w_1(n_e)$, see eq. (\ref{eq:FIO_g1}), one gets an operator which is linear in terms of $f$ and such that $\mathcal{L}_1 (n_e,n_e) = \mathcal{T}(n_e)$. Furthermore, the operator
\begin{equation}\label{eq:def_L2}
\mathcal{L}_2(f,n_e)(\lambda,\mathbf{d},\mathbf{s}) = \int_\Omega \int_{\Omega\setminus\{\mathbf{x}\}} \mathcal{W}_2(n_e)(\mathbf{x},\mathbf{y},\mathbf{d},\mathbf{s})
f(\mathbf{x}) f(\mathbf{y}) \delta(\lambda - \Psi(\mathbf{y},\mathbf{x},\mathbf{d},\mathbf{s})) \mathrm{d}\mathbf{y} \mathrm{d}\mathbf{x}
\end{equation}
is linear w.r.t. the six-dimensional $f(\mathbf{x}) f(\mathbf{y})$ and satisfies $\mathcal{L}_2 (n_e,n_e) = \mathcal{S}(n_e)$ given in eq. (\ref{eq:def_S}).\\

The approximations of $\mathcal{T}$ and $\mathcal{S}$ by $\mathcal{L}_1$ and $\mathcal{L}_2$ are 
justified by Theorems \ref{theo:continuity_T} and \ref{theo:continuity_S}. These two theorems require the following Lemmas.

\begin{lemma}\label{theo:1to1} Let $\mathbf{s}$ and $\mathbf{d}$ be fixed. Then, for every $\mathbf{x}\in \mathbb{R}^3\setminus\{\mathbf{s}+t(\mathbf{d}-\mathbf{s}), t\in \mathbb{R}\}$, there exists a unique $\omega$ such that $\mathbf{x} \in \mathfrak{T}(\omega,\mathbf{s},\mathbf{d})$. In other words, the spindle tori $\mathfrak{T}(\omega,\mathbf{s},\mathbf{d})$ with $\omega \in (0,\pi)$ are one-to-one with $\mathbb{R}^3\setminus\{\mathbf{s}+t(\mathbf{d}-\mathbf{s}), t\in \mathbb{R}\}$. 
\end{lemma}

\begin{proof}
See Appendix.
\end{proof}

\begin{lemma}\label{theo:1to1_intersection} Let $\mathbf{s}$, $\mathbf{x}$ and $\mathbf{d}$ be fixed. Then, for every $\mathbf{y}\in \mathbb{R}^3\setminus\{\mathbf{x}+t(\mathbf{d}-\mathbf{x}), t\in \mathbb{R}\}$, there exists a unique $\omega_1$ such that $\mathbf{y} \in \mathfrak{T}(\omega_2(\omega_1),\mathbf{d},\mathbf{x}) \cap \mathfrak{C}(\omega_1,\mathbf{x},\mathbf{s})$. 
\end{lemma}
\begin{proof}
Using the definition of the intersection given in eq. (\ref{eq:coordinates_intersection}), the radius is defined as
$$
r_\cap := \Vert \mathbf{d}-\mathbf{x}\Vert \left( z_\cap - \frac{\sqrt{1-z_\cap^2}}{\tan \omega_2}, \right)
$$
where 
$$
\cot \omega_2(\omega_1) = \frac{\lambda(E)-\cos \omega_1}{\sqrt{1-(\lambda(E)-\cos \omega_1)^2}}.
$$
This latter expression is a bijection with respect to $\lambda(E)$ with image $(0,\infty)$. Therefore, the proof is analogous to the proof of Lemma \ref{theo:1to1}.
\end{proof}

\begin{theorem}\label{theo:continuity_T} Let $f \in L_2(\Omega)$ with $\Omega\subset \mathbb{R}^3$. Since $C^\infty(\Omega)$ is dense in $L_2(\Omega)$, there exists a sequence $(f_n)_n \in C^\infty(\Omega)$ such that $f_n \to f$ as $n \to \infty$.  We denote by $\mathcal{T}_\mathbf{d}$ and $\mathcal{L}_{1,\mathbf{d}}$ the restrictions of $\mathcal{T}$ and $\mathcal{L}_1$ to one detector $\mathbf{d}$ and a given source $\mathbf{s}$. Then, it holds
$$
\lim_{n\to \infty} \Vert \mathcal{T}_\mathbf{d}(f_n) - \mathcal{L}_{1,\mathbf{d}}(f,f_n)\Vert_{L_2(\mathbb{R})} = 0.
$$
\end{theorem}

\begin{proof}
See Appendix.
\end{proof}

\begin{theorem}\label{theo:continuity_S} 
Let $f \in L_2(\Omega)\cap L_\infty(\Omega)$ with $\Omega\subset \mathbb{R}^3$. Since $C^\infty(\Omega)$ is dense in $L_2(\Omega)$, it exists a sequence $(f_n)_n \in C^\infty(\Omega)$ such that $f_n \to f$ as $n \to \infty$.  We denote by $\mathcal{S}_\mathbf{d}$ and $\mathcal{L}_{2,\mathbf{d}}$ the restrictions of $\mathcal{S}$ and $\mathcal{L}_2$ to one detector $\mathbf{d}$ and a given source $\mathbf{s}$. Then, it holds
$$
\lim_{n\to \infty} \Vert \mathcal{S}_\mathbf{d}(f_n) - \mathcal{L}_{2,\mathbf{d}}(f,f_n)\Vert_{L_2(\mathbb{R})} = 0.
$$
\end{theorem}
\begin{proof}
Analogous to the proof of Theorem \ref{theo:continuity_T}.
\end{proof}

These theorems justify why the mapping $\mathcal{T}(n_e)$ (resp. $\mathcal{S}(n_e)$) can be approximated by $\mathcal{L}_1(f,n_e)$ (resp. $\mathcal{L}_2(f,n_e)$) for $f \in L_2(\Omega)\cap L_\infty(\Omega)$ and $n_e \in C^\infty(\Omega)$ such that $\Vert f - n_e \Vert_{L_2(\Omega)} \leq \epsilon \ll 1$.

In the following, we study $\mathcal{L}_1$ and $\mathcal{L}_2$ instead of $\mathcal{T}$ and $\mathcal{S}$ using the theory of Fourier integral operators.

\subsection{FIO and smoothness properties of $\mathcal{L}_1$ and $\mathcal{L}_2$}

We now derive the main results of this section which state the smoothness properties of $\mathcal{L}_1$ and $\mathcal{L}_2$.

\begin{theorem}\label{theo:FIO_L1} Let $n_e \in C^\infty(\Omega)$ given with $\Omega$ an open subset of $\mathbb{R}^3$. Then $\mathcal{L}_1$, defined in eq. (\ref{eq:def_L1}), belongs to $I^{-1}(\mathbb{R} \times \mathbb{D},\Omega)$ and is thus a Fourier integral operator of order $-1$ with respect to $f$ when the source $\mathbf{s}$ is fixed.
\end{theorem}
\begin{proof}
Using the Fourier representation of the Dirac distribution, one gets 
$$
\mathcal{L}_1(f,n_e)(p,\mathbf{d},\mathbf{s}) = \frac{1}{\sqrt{2\pi}} \int_{\Omega}\int_\mathbb{R} \mathcal{W}_1 (n_e) (\mathbf{x},\mathbf{d},\mathbf{s}) \ f(\mathbf{x}) \ e^{-i\sigma(p-\phi(\mathbf{x},\mathbf{d},\mathbf{s}))} \ \mathrm{d}\sigma \mathrm{d}\mathbf{x}.
$$
Taking into account that $\mathbf{s}$ is fixed and defining the phase
$$
\Phi(p,\mathbf{d},\mathbf{x},\sigma) = \sigma(p - \phi(\mathbf{x},\mathbf{d},\mathbf{s})),
$$
one needs to prove that $(\partial_{p,\mathbf{d}} \Phi, \partial_\sigma \Phi)$ and $(\partial_\mathbf{x} \Phi, \partial_\sigma \Phi)$ do not vanish for all $(p,\mathbf{d},\mathbf{x},\sigma) \in \mathbb{R} \times \mathbb{D} \times \Omega \times \mathbb{R}\setminus\{0\}$, that the phase is non-degenerate and positive homogeneous of degree 1 w.r.t. $\sigma$. The homogeneity is clear. We further obtain
$$
\Sigma_\Phi = \{(p,\mathbf{d},\mathbf{x},\sigma) \in \mathbb{R} \times \mathbb{D} \times \Omega \times \mathbb{R} \setminus \{0\} \ | \ \partial_\sigma \Phi = 0\}
$$
with
$$
\partial_\sigma \Phi = (p - \phi(\mathbf{x},\mathbf{d},\mathbf{s})) \mathrm{d}\sigma.
$$
One obtains thus that the phase is nondegenerate since 
$$
\partial_{p,\mathbf{d},\mathbf{x}} (\partial_\sigma \Phi) = \mathrm{d}\sigma (\mathrm{d}p - \partial_\mathbf{d} \phi(\mathbf{x},\mathbf{d},\mathbf{s}) - \partial_\mathbf{x} \phi(\mathbf{x},\mathbf{d},\mathbf{s})) \neq 0.
$$
Now we need to prove that $(\partial_{p,\mathbf{d}} \Phi, \partial_\sigma \Phi)$ and $(\partial_\mathbf{x} \Phi, \partial_\sigma \Phi)$ do not vanish for all $(p,\mathbf{d},\mathbf{x},\mathbf{y},\sigma) \in \mathbb{R} \times \mathbb{D} \times \Omega_2 \times \mathbb{R}\setminus\{0\}$. For the first case, one has
$$
\partial_{p,\mathbf{d}} \Phi = \sigma ( \mathrm{d}p - \partial_\mathbf{d} \phi(\mathbf{x},\mathbf{d},\mathbf{s}) ).
$$
Since $\sigma \neq 0$, this derivative never vanishes. It remains to show that 
$(\partial_\mathbf{x} \Phi, \partial_\sigma \Phi)$ does not vanish. A straightforward calculation gives 
\begin{eqnarray}\label{eq:gradient_phase_phi}
\nabla_{\mathbf{x}}\phi(\mathbf{x},\mathbf{d},\mathbf{s})  
&= \frac{(1- \rho(\mathbf{x}) \kappa(\mathbf{x}) )}{\Vert \mathbf{x}-\mathbf{s}\Vert (1-\kappa^2(\mathbf{x})^{3/2}} \frac{\mathbf{d}-\mathbf{s}}{\Vert \mathbf{d}-\mathbf{s} \Vert} \nonumber\\
&- \frac{1}{\Vert \mathbf{x}-\mathbf{s}\Vert \sqrt{1-\kappa^2(\mathbf{x})}} 
\left( \rho(\mathbf{x}) + \kappa(\mathbf{x}) \frac{1-\rho(\mathbf{x})\kappa(\mathbf{x})}{1-\kappa^2(\mathbf{x})}
\right) \frac{\mathbf{x}-\mathbf{s}}{\Vert\mathbf{x}-\mathbf{s}\Vert}.
\end{eqnarray}
After some computations, the norm of $\nabla_{\mathbf{x}}\phi(\mathbf{x},\mathbf{d},\mathbf{s})$ can then be written as
$$
\Vert \nabla_{\mathbf{x}}\phi(\mathbf{x},\mathbf{d},\mathbf{s})\Vert^2 = \frac{1-2\; \rho(\mathbf{x})\kappa(\mathbf{x})+ \rho^2(\mathbf{x})}{(1-\kappa^2(\mathbf{x}))^2 \Vert \mathbf{x}-\mathbf{s}\Vert^2}.
$$
The numerator has only complex roots and thus never cancels out while $\kappa(\mathbf{x}) = 1$ corresponds to the degenerate case of the torus which is excluded (primary radiation).  In consequence, $(\partial_\mathbf{x} \Phi, \partial_\sigma \Phi)$ is never zero and $\mathcal{L}_1$ is a FIO.

Regarding the symbol, the definition of $\mathcal{L}_1$ in eq. (\ref{eq:def_L1}) leads to the symbol
$$
\mathcal{W}_1 (n_e)(\mathbf{x},\mathbf{d},\mathbf{s}) = \Vert\nabla_\mathbf{x} \phi (\mathbf{x},\mathbf{d},\mathbf{s}) \Vert  A_{E_0}(\mathbf{s},\mathbf{x}) \cdot A_{E_\omega}(\mathbf{x},\mathbf{d}).
$$
The phase $\phi$ is a smooth function with non-vanishing derivatives and since $n_e$ is assumed to be $C^\infty$ smooth and compactly supported, its integrals along the traveling paths, $\mathbf{s}\to\mathbf{x}\to\mathbf{d}$, are also $C^\infty$. The symbol is similar to the one of the attenuated Radon transforms, see \cite{Kuchment}. Therefore, the symbol of $\mathcal{L}_1$ satisfies eq. (\ref{eq:cond_amplitude}) and is of order 0. Hence, following on from Theorem \ref{theo:Hormander}, $\mathcal{L}_1(\cdot,n_e)$ is thus a FIO of order 
$$
m = \frac{1}{2} - \frac{3+3}{4} = -1.
$$  
\end{proof}

We now give the equivalent property for the second-order and its approximation the operator $\mathcal{L}_2$. 

\begin{theorem}\label{theo:FIO_L2} Let the source $\mathbf{s}$ be fixed and $n_e \in C^\infty(\Omega)$ given with $\Omega$ an open subset of $\mathbb{R}^3$.
Then, $\mathcal{L}_2$, defined in eq. (\ref{eq:def_L2}), belongs to $I^{-\frac74}(\mathbb{R} \times \mathbb{D},\Omega_2)$ with 
$$
\Omega_2 = \left\{(\mathbf{x},\mathbf{y}) \in \Omega \times \Omega \ | \  \mathbf{y} \notin \mathfrak{N}(\mathbf{x})\right\},
$$
where $\mathfrak{N}(\mathbf{x})$ denoting a neighborhood around $\mathbf{x}$. $\mathcal{L}_2$ is thus a FIO of order $-\frac{7}{4}$ with respect to $f$.
\end{theorem}

\textbf{Remark:} The construction of $\Omega_2$ is not constraining in practice as two successive scattering events cannot occur at the same location.

\begin{proof}
By analogy with $\mathcal{L}_1$, and since $\mathbf{s}$ is fixed, one can define the phase
$$
\Upsilon(\lambda,\mathbf{d},\mathbf{y},\mathbf{x},\sigma) = \sigma(\lambda - \Psi(\mathbf{y},\mathbf{x},\mathbf{d},\mathbf{s})),
$$
and one needs to prove that $(\partial_{\lambda,\mathbf{d}} \Upsilon, \partial_\sigma \Upsilon)$ and $(\partial_{\mathbf{y},\mathbf{x}} \Upsilon, \partial_\sigma \Upsilon)$ do not vanish for all $(\lambda,\mathbf{d},\mathbf{x},\mathbf{y},\sigma) \in (0,2) \times \mathbb{D} \times \Omega_2 \times \mathbb{R}\setminus\{0\}$, that the phase is non-degenerate and positive homogeneous of degree 1 w.r.t. $\sigma$. The homogeneity is clear. We further obtain
$$
\Sigma_\Upsilon = \{(\lambda,\mathbf{d},\mathbf{y},\mathbf{x},\sigma) \in (0,2) \times \mathbb{D} \times \Omega \times \Omega \times \mathbb{R} \setminus \{0\} \ | \ \partial_\sigma \Upsilon = 0\}
$$
with
$$
\partial_\sigma \Upsilon = (\lambda - \Psi(\mathbf{y},\mathbf{x},\mathbf{d},\mathbf{s})) \mathrm{d}\sigma.
$$
One obtains thus that the phase is nondegenerate since 
$$
\partial_{\lambda,\mathbf{d},\mathbf{y},\mathbf{x}} (\partial_\sigma \Phi) = \mathrm{d}\sigma (\mathrm{d}p - \partial_\mathbf{d} \Psi(\mathbf{y},\mathbf{x},\mathbf{d},\mathbf{s})  - \partial_\mathbf{y} \Psi(\mathbf{y},\mathbf{x},\mathbf{d},\mathbf{s}) - \partial_\mathbf{x} \Psi(\mathbf{y},\mathbf{x},\mathbf{d},\mathbf{s})) \neq 0.
$$
Now we need to prove that $(\partial_{\lambda,\mathbf{d}} \Upsilon, \partial_\sigma \Upsilon)$ and $(\partial_{\mathbf{y},\mathbf{x}} \Upsilon, \partial_\sigma \Upsilon)$ do not vanish for all $(\lambda,\mathbf{d},\mathbf{x},\mathbf{y},\sigma) \in (0,2)\times \mathbb{D} \times \Omega_2 \times \mathbb{R}\setminus\{0\}$. For the first case, one has
$$
\partial_{\lambda,\mathbf{d}} \Upsilon = \sigma ( \mathrm{d}\lambda - \partial_\mathbf{d} \Psi(\mathbf{y},\mathbf{x},\mathbf{d},\mathbf{s}) ).
$$
Since $\sigma \neq 0$, this derivative never vanishes. It remains to show that 
$(\partial_{\mathbf{y},\mathbf{x}} \Upsilon, \partial_\sigma \Upsilon)$ does not vanish. Computing the gradient w.r.t $\mathbf{y}$ of both components in eq. (\ref{eq:phase_g2}) gives
\begin{equation}\label{eq:nabla_psi}
\nabla_{\mathbf{y}} \psi(\mathbf{y},\mathbf{x},\mathbf{s}) = \frac{1}{\Vert \mathbf{y}-\mathbf{x}\Vert} \left( \frac{(\mathbf{x}-\mathbf{s})}{\Vert \mathbf{x}-\mathbf{s}\Vert} - \left(\frac{ (\mathbf{x}-\mathbf{s})}{\Vert \mathbf{x}-\mathbf{s}\Vert} \cdot  \frac{(\mathbf{y}-\mathbf{x}) }{\Vert \mathbf{y}-\mathbf{x}\Vert} \right) \frac{(\mathbf{y}-\mathbf{x})}{\Vert \mathbf{y}-\mathbf{x}\Vert} \right)
\end{equation}
and
\begin{equation}\label{eq:nabla_cos_phi}
\nabla_{\mathbf{y}} \cos \left( \cot^{-1} \phi(\mathbf{y},\mathbf{d},\mathbf{x})\right) =
\frac{\nabla_{\mathbf{y}} \phi(\mathbf{y},\mathbf{d},\mathbf{x})}{(1+\phi(\mathbf{y},\mathbf{d},\mathbf{x})^2)^{3/2}} 
\end{equation}
with
\begin{eqnarray*}
\nabla_{\mathbf{y}}\phi(\mathbf{y},\mathbf{d},\mathbf{x})  
&= \frac{(1- \rho(\mathbf{y}) \kappa(\mathbf{y}) )}{\Vert \mathbf{y}-\mathbf{x}\Vert (1-\kappa^2(\mathbf{y})^{3/2})} \frac{\mathbf{d}-\mathbf{x}}{\Vert \mathbf{d}-\mathbf{x} \Vert} \nonumber\\
&- \frac{1}{\Vert \mathbf{y}-\mathbf{x}\Vert \sqrt{1-\kappa^2(\mathbf{y})}} 
\left( \rho(\mathbf{y}) + \kappa(\mathbf{y}) \frac{1-\rho(\mathbf{y})\kappa(\mathbf{y})}{1-\kappa^2(\mathbf{y})}
\right) \frac{\mathbf{y}-\mathbf{x}}{\Vert\mathbf{y}-\mathbf{x}\Vert}.
\end{eqnarray*}

After computations, we observe that $\nabla_{\mathbf{y}} \Psi \cdot (\mathbf{y} - \mathbf{x})$ leads to
$$
\nabla_{\mathbf{y}} \psi(\mathbf{y},\mathbf{x},\mathbf{s}) \cdot (\mathbf{y} - \mathbf{x})= 0
$$
while 
$$
\nabla_{\mathbf{y}} \cos \left( \cot^{-1} \phi(\mathbf{y},\mathbf{d},\mathbf{x})\right)\cdot (\mathbf{y} - \mathbf{x}) = - \frac{\rho(\mathbf{y})}{\sqrt{1-\kappa^2(\mathbf{y})}}.
$$
This latter component is then never zero if $\rho(\mathbf{y}) \neq 0$ or equivalently when $\mathbf{y} \neq \mathbf{x}$ which corresponds to our construction of $\Omega_2$.
Therefore, $(\nabla_\mathbf{y}\Psi,\nabla_\mathbf{x}\Psi)\neq\mathbf{0}$  and $\Psi$ defines a phase.

To interpret $\mathcal{L}_2$ as a FIO, we embed the variable $(\mathbf{x},\mathbf{y})$ into $\mathbf{z} \in \Omega_2$. Taking $\bar{f}(\mathbf{z}) = f(\mathbf{x}) f(\mathbf{y})$, it yields
\begin{eqnarray*}
\mathcal{L}_2(\bar{f},n_e)(\lambda,\mathbf{d},\mathbf{s}) 
&=& \int_{\Omega_2} \mathcal{W}_2(n_e)(\mathbf{z},\mathbf{d},\mathbf{s})
\bar{f}(\mathbf{z}) \delta(\lambda - \bar{\Psi}(\mathbf{z},\mathbf{d},\mathbf{s})) \mathrm{d}\mathbf{z}\\
&=& \frac{1}{\sqrt{2\pi}}\int_{\Omega_2} \int_\mathbb{R} \mathcal{W}_2(n_e)(\mathbf{z},\mathbf{d},\mathbf{s})
\bar{f}(\mathbf{z}) e^{-i\sigma(\lambda - \bar{\Psi}(\mathbf{z},\mathbf{d},\mathbf{s}))} \mathrm{d}\sigma \mathrm{d}\mathbf{z}\\
\end{eqnarray*}
with $\bar{\Psi}$ the appropriate embedded version of $\Psi$ which inherits its properties of phase.

Regarding the symbol, the definition of $\mathcal{L}_2$ in eq. (\ref{eq:def_L2}) leads to the symbol
$$
\mathcal{W}_2 (n_e)(\mathbf{x},\mathbf{y},\mathbf{d},\mathbf{s}) = \Vert\nabla_\mathbf{\mathbf{x},\mathbf{y}} {\Psi}(\mathbf{y},\mathbf{x},\mathbf{d},\mathbf{s}) \Vert  A_{E_0}(\mathbf{s},\mathbf{x}) \cdot A_{E_{\omega_1}}(\mathbf{x},\mathbf{y}) \cdot A_{E_{\omega_2}}(\mathbf{y},\mathbf{d}).
$$
The phase $\bar{\Psi}$ is a smooth function with non-vanishing derivatives and since $n_e$ is assumed to be $C^\infty$ smooth and compactly supported, its integrals along the traveling paths, $\mathbf{s}\to\mathbf{x}\to\mathbf{y}\to\mathbf{d}$, are also $C^\infty$. As for $\mathcal{L}_1$, the symbol is similar to the one of the attenuated Radon transforms, see \cite{Kuchment}. Therefore, the symbol of $\mathcal{L}_2$ satisfies eq. (\ref{eq:cond_amplitude}) and is of order 0.
Given Theorem \ref{theo:Hormander}, $\mathcal{L}_2$ is a FIO of order 
$$
m = \frac{1}{2} - \frac{6+3}{4} = -\frac{7}{4}
$$ 
which ends the proof.
\end{proof}

It follows that the approximation of the second-order scattering $g_2$, $\mathcal{L}_2$, is substantially smoother than the approximation to the first-order, $\mathcal{L}_1$. This result is exploited to extract the contours of $f$ ($L_2$-approximation of $n_e$) using reconstruction techniques built to solve $g_1 = \mathcal{T}(n_e)$ in Section \ref{sec:simu_MC}. We now give more details about the smoothness properties using Sobolev spaces. Theorems \ref{theo:Hormander} requires that the natural projections $\Pi_{\mathcal{L}_1}$ and $\Pi_{\mathcal{L}_2}$ are immersions. The property depends on the domain of definition $\mathbb{D}$ and will not be true for random sets of detectors.

In the following Lemma, we provide a condition for $\mathbb{D}$ when the source is fixed in order to get that $\Pi_{\mathcal{L}_1}$ is an immersion.

\begin{lemma}\label{lemma:immersion_sphere}
Let the source $\mathbf{s}$ be fixed and $\mathbf{d}$ be written as 
$$
\mathbf{d}(\alpha,\beta) = \mathbf{s} +  t(\alpha) 
\left(\begin{array}{c}
\sin \alpha \cos \beta\\ 
\sin \alpha \sin \beta\\ 
\cos \alpha
\end{array} \right)
\qquad (\alpha,\beta) \in [0,\alpha_{\mathrm{max}}] \times [0,2\pi].
$$
$\alpha_{\mathrm{max}}$ is typically set to $\pi$ or $\pi/2$. Let $\Omega \subset \mathbb{R}^3$ be supported strictly inside the convex domain generated by the surface $\mathbb{D}$ such that
\begin{equation}\label{eq:condition_immersion}
(\mathbf{d}-\mathbf{s}) \cdot (\mathbf{d} - \mathbf{x}) \neq 0, \quad \forall \ (\mathbf{d},\mathbf{x}) \in \mathbb{D}\times \Omega. 
\end{equation}
If for all $\alpha \in  [0,\alpha_{\mathrm{max}}]$ and $\mathbf{x} \in \Omega$ it holds
$$
 \Vert \mathbf{x} - \mathbf{s}\Vert \neq t(\alpha) \kappa(\mathbf{x}) + \partial_\alpha \kappa(\mathbf{x}) \partial_\alpha t(\alpha)
\quad \text{with} \quad 
\partial_\alpha \kappa(\mathbf{x}) =  \frac{\mathbf{x}-\mathbf{s}}{\Vert \mathbf{x}-\mathbf{s}\Vert} \cdot
\left(\begin{array}{c}
\cos \alpha \cos \beta\\ 
\cos \alpha \sin \beta\\ 
-\sin \alpha
\end{array} \right) ,
$$
with $\kappa$ defined in eq. (\ref{eq:def_kappa_rho}), then, 
\begin{equation}\label{eq:jacobian}
h(\mathbf{x},\mathbf{d}):=\mathrm{det}(\nabla_\mathbf{x} \phi, \partial_\alpha \nabla_\mathbf{x} \phi, \partial_\beta \nabla_\mathbf{x} \phi)(\mathbf{x},\mathbf{d},\mathbf{s}) \neq 0, \quad \forall \ (\mathbf{x},\mathbf{d})\in \Omega \times \mathbb{D}.
\end{equation}
\end{lemma}
\begin{proof}
See Appendix
\end{proof}

The property
$$
 \Vert \mathbf{x} - \mathbf{s}\Vert \neq t(\alpha) \kappa(\mathbf{x}) + \partial_\alpha \kappa(\mathbf{x}) \partial_\alpha t(\alpha)
$$
is obviously satisfied when $t(\alpha)$ is constant. 
For, $t(\alpha) = \cos \alpha$, the condition leads, after computations, to the condition $\mathbf{x}\in\mathbb{D}$ which is excluded. 

This condition along with equation (\ref{eq:condition_immersion}) delivers an important information on how to build the scanner and where to locate the object under study. We discuss the consequences in Section \ref{sec:simu_MC}.

\begin{corollary}\label{theo:smoothness_L1} Under the conditions stated in Lemma  \ref{lemma:immersion_sphere}, the operator $\mathcal{L}_1~:~H_c^\nu(\Omega)~\to ~ H^{\nu+1}_{loc}(\mathbb{R}\times\mathbb{D})$ is continuous.
\end{corollary}
\begin{proof}
According to Lemma \ref{lemma:immersion_sphere}, $\Phi$ is an immersion and thus $\mathcal{L}_1$ has the desired property using  Theorems \ref{theo:Hormander} and \ref{theo:FIO_L1}.
\end{proof}

\begin{corollary}\label{theo:smoothness_L2}
Let the source $\mathbf{s}$ be fixed and the vector $\mathbf{d} \in \mathbb{D}$ be written as 
$$
\mathbf{d}(\alpha,\beta) = \mathbf{s} + t(\alpha) 
\left(\begin{array}{c}
\sin \alpha \cos \beta\\ 
\sin \alpha \sin \beta\\ 
\cos \alpha
\end{array} \right)
\qquad (\alpha,\beta) \in [0,\alpha_{\mathrm{max}}] \times [0,2\pi].
$$
Let $\Omega$ be supported strictly inside the convex domain generated by the detector surface such that
\begin{description}
\item[(i)]
$
(\mathbf{d}-\mathbf{x}) \cdot (\mathbf{y} - \mathbf{x}) \neq 0,
$
\item[(ii)] 
$
 \Vert \mathbf{y} - \mathbf{x}\Vert \neq \Vert \mathbf{d} - \mathbf{x}\Vert \kappa(\mathbf{y}) + \partial_\alpha \kappa(\mathbf{y}) \partial_\alpha \Vert \mathbf{d} - \mathbf{x}\Vert 
$
\item[(iii)]
\begin{equation}\label{eq:condition_immersion_L2}
\nabla_{\mathbf{x},\mathbf{y}} \psi \notin \mathrm{span}(\partial_\alpha \nabla_{\mathbf{x},\mathbf{y}} \phi, \partial_\beta \nabla_{\mathbf{x},\mathbf{y}} \phi) - \frac{1}{(1+\phi^2)^{3/2}} \nabla_{\mathbf{x},\mathbf{y}} \phi
\end{equation}
for all $(\mathbf{x},\mathbf{y},\mathbf{d})\in\Omega_2 \times \mathbb{D}$
\end{description}
then $\mathcal{L}_2: H_c^\nu(\Omega_2) \to H^{\nu+7/4}_{loc}(\mathbb{R}\times\mathbb{D})$ is continuous. If (iii) is not satisfied,\\ then $\mathcal{L}_2: H_c^\nu(\Omega_2) \to H^{\nu+5/4}_{loc}(\mathbb{R}\times\mathbb{D})$ is continuous.
\end{corollary}
\begin{proof}
We need to prove that the matrix $(\nabla_{\mathbf{x},\mathbf{y}} \Psi,\partial_\alpha \nabla_{\mathbf{x},\mathbf{y}} \Psi, \partial_\beta \nabla_{\mathbf{x},\mathbf{y}} \Psi)$ is full-rank, \textit{i.e.} is of rank 3 if (i) and (ii) are satisfied, see Lemma \ref{lemma:immersion_sphere}. Due to the definition of the phase $\Psi$, we observe that the upper part of this matrix satisfies
$$
(\nabla_{\mathbf{x}} \Psi,\partial_\alpha \nabla_{\mathbf{x}} \Psi, \partial_\beta \nabla_{\mathbf{x}} \Psi)
=  (\nabla_{\mathbf{x}} \psi,0,0) + \frac{1}{(1+\phi^2)^{3/2}} (\nabla_{\mathbf{x}} \phi,\partial_\alpha \nabla_{\mathbf{x}} \phi, \partial_\beta \nabla_{\mathbf{x}} \phi).
$$
We know that $(\nabla_{\mathbf{x}} \phi,\partial_\alpha \nabla_{\mathbf{x}} \phi, \partial_\beta \nabla_{\mathbf{x}} \phi)$ is full rank. Therefore, the matrix $(\nabla_{\mathbf{x},\mathbf{y}} \Psi,\partial_\alpha \nabla_{\mathbf{x},\mathbf{y}} \Psi, \partial_\beta \nabla_{\mathbf{x},\mathbf{y}} \Psi)$  is full rank if and only if equation (\ref{eq:condition_immersion_L2}) is satisfied. In this case, $\Pi_{\mathcal{L}_2}:C_{\mathcal{L}_2} \to \mathbb{R}\times \mathbb{D}$ is an immersion and the desired result follows when combining Theorems \ref{theo:Hormander} and \ref{theo:FIO_L2}. If equation (\ref{eq:condition_immersion_L2}) is not satisfied then the smoothness propertes can drop. Since the vectors $(\partial_\alpha \nabla_{\mathbf{x},\mathbf{y}} \phi, \partial_\beta \nabla_{\mathbf{x},\mathbf{y}} \phi)$ are linearly independant, $\mathrm{dim}(\mathrm{ker}(\Pi_{\mathcal{L}_2})) = 1$ and following from \cite{Hormander} the smoothness properties decreases of a degree $\frac12$ which ends the proof.
\end{proof}

Conditions (i) and (ii) are satisfied when the object is rather small and close from the center of $\mathbb{D}$. Therefore, the continuity of $\mathcal{L}_2$ here depends essentially on the condition given by eq. (\ref{eq:condition_immersion_L2}). The condition is difficult to verify numerically due to the size of the vectors (these are elements of $\mathbb{R}^6$) but dismisses \textit{only} a two-dimensional surface within a six-dimensional space. In the worst case ((iii) not satisfied), $\mathcal{L}_2$ remains smoother than $\mathcal{L}_1$ but only by an order $\frac14$.

\subsection{Overview of the analysis}

The model operator $\mathcal{T}$ and $\mathcal{S}$ describing the measurement process for $g_1$ and $g_2$ can be seen as approximations of the FIO $\mathcal{L}_1$ and $\mathcal{L}_2$ assuming the sought-for electron density $n_e$ is smooth. While $\mathcal{L}_1$ is of order -1, $\mathcal{L}_2$ is of order $-\frac74$, at least $-\frac54$, quantifying the degree of integration involved by $g_1$ and $g_2$. This result indicates that $\mathcal{L}_2$ smoothes significantly more than $\mathcal{L}_1$ and then that this latter brings a richer information about the singularities. While the study of the wave-front sets will be the subject of further research, this motivates the development of reconstruction techniques focused on $g_1$ and using appropriate differentiation orders in order to discard $g_2$. This is proposed in the following section.

We conclude this section by stating an intuitive conjecture regarding the extension to the whole spectrum, \textit{i.e.} to all scattering of higher orders.  By analogy to our derivations in Section \ref{sec:model}, the scattering of order $k$ ($k>2$) will rely on the relation 
$$
\sum_{i=1}^k \cos \omega_i = k - mc^2 \left( \frac{1}{E} - \frac{1}{E_0} \right)
$$
with $\omega_i$ the $i^{th}$ scattering angle. Unfortunately, the geometry of the scattering events becomes harder to model or even to implement. However, we expect the principle we developed for the second scattering to extend to higher orders since each additional scattering will add a new \textit{layer} of intermediary sources. This extension is expressed in the following statement. 

\begin{hypothesis}\label{hypo}
When focusing on the Compton scattering, the measured spectrum can be decomposed into the different scattered radiation of $i^{th}$ order,
$$
\mathrm{Spec}(E,\mathbf{d},\mathbf{s}) = \sum_{i=1}^\infty g_i(E,\mathbf{d},\mathbf{s})
$$
with $g_i$ approximated by the FIO $\mathcal{L}_i (\cdot,n_e)$ of order $-(3i+1)/4$. Intuitively, the more scattering events, the smoother the corresponding part within the spectrum is. 
\end{hypothesis}

\section{Contour reconstruction and simulation results}\label{sec:simu_MC}

Solving the inverse problem 
$$
(\mathcal{T}+\mathcal{S})(n_e) = (g_1 + g_2 + \eta) =: \mathrm{Spec},
$$
is very challenging due to the nonlinearity with respect to $n_e$ but also the complexity of $\mathcal{S}$, see Section \ref{sec:model}.  Non-linear iterative techniques such as the Landweber iteration or the Kaczmarz's method could provide satisfactory reconstruction of $n_e$ but at the cost of a tremendous computation time. Deep learning techniques could be used but the lack of large dataset prevents any training step. In this section, we propose to circumvent these two issues by,
\begin{itemize}
\item first approximating $\mathcal{T}(n_e)$ and $\mathcal{S}(n_e)$ by $\mathcal{L}_1(f,n_e)$ and $\mathcal{L}_2 (f,n_e)$ respectively,
\item then focusing on the inverse problem to find $f$ from $\mathcal{L}_1(f,n_e) = g_1$ which provides a richer information than the second-order scattering, approximated by the inverse problem $\mathcal{L}_2 (f,n_e) = g_2$, see Section \ref{sec:smooth}.  
\end{itemize}
This approach will be valid when $n_e \in C^\infty(\Omega)$ and $f\in L_2(\Omega)$ such that $n_e$ stands for an approximation of $f$. A first strategy was proposed in \cite{RH_2018} for extracting the contours of $f$ from $g_1= \mathcal{L}_1(f,n_e)$ using a filtered-backprojection-type formula. The results derived in Section \ref{sec:smooth} suggests that the same approach could be applied in the more realistic case of multiple scattering. We propose thus to apply on the spectrum, $\mathrm{Spec}$, approximated by $(\mathcal{L}_1 + \mathcal{L}_2)(f,n_e)$, the filtered-backprojection type inversion formula 

\begin{figure}[!t]\centering
\includegraphics[width=\linewidth]{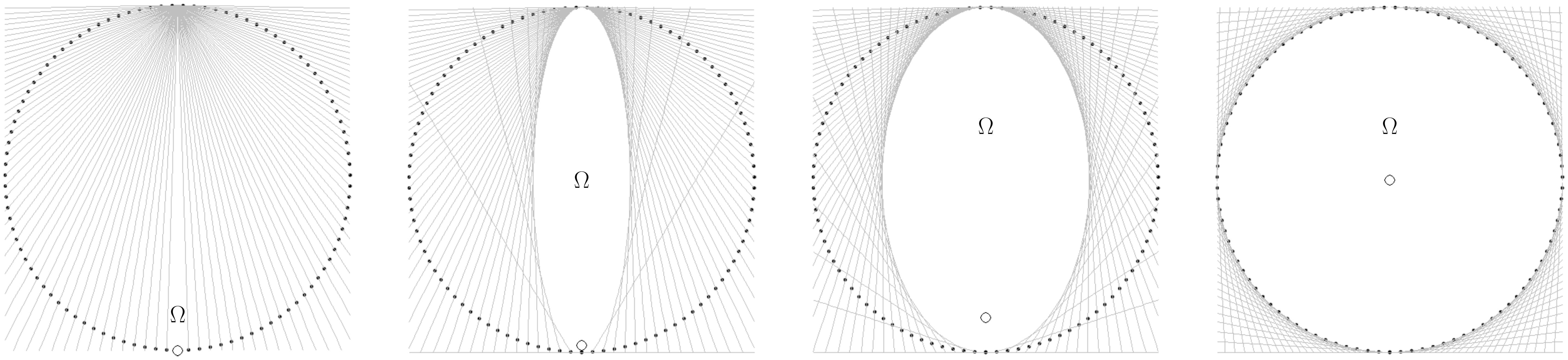}
\caption{Allowed support of the object, $\Omega$, for different positions of the source w.r.t. to the sphere of detectors, $\mathbb{D}$, here on a slice. The black dots depict the different detectors, the small circle is the source, while the gray lines represent the conditions in eq. (\ref{eq:condition_immersion}). 
\label{fig:immersion_source_sphere}}
\end{figure}

\begin{equation}\label{eq:FBP_CSI}
\tilde{f} = \mathcal{B} \partial_p^2 (\mathrm{Spec})
\end{equation}
with $\partial_p^2$ the second derivative w.r.t the first variable and the backprojection operator
$$
\mathcal{B} g(\mathbf{x})  = \int_{\mathbb{D}} \left(  \mathcal{W}_1(n_e^{prior}) (\mathbf{x};\phi(\mathbf{x},\mathbf{d},\mathbf{s}),\mathbf{d},\mathbf{s}) \right)^{-1} \; h(\mathbf{x},\mathbf{d}) g(\phi(\mathbf{x},\mathbf{d},\mathbf{s}),\mathbf{d}) \mathrm{d} \mathbf{d}
$$
in which 
$$
h(\mathbf{x},\mathbf{d}) = \left\vert \mathrm{det} \left(
\begin{array}{c}
\nabla_\mathbf{x} \phi(\mathbf{x},\mathbf{d},\mathbf{s}) \\
\partial_{\theta_1}\nabla_\mathbf{x} \phi(\mathbf{x},\mathbf{d},\mathbf{s})\\
\partial_{\theta_2}\nabla_\mathbf{x} \phi(\mathbf{x},\mathbf{d},\mathbf{s})
\end{array}
\right)
\right\vert
$$
and $n_e^{prior}$ stands for a prior knowledge on the electron density $n_e$. Equivalently, the extraction of the contours of $f$, noted $f_c$, is then achieved by the formula
\begin{equation}\label{eq:nabla_FBP_CSI}
\tilde{f}_c = \nabla_\mathbf{x} \mathcal{B} \partial_p^2 (\mathrm{Spec}).
\end{equation}
We note that $\mathcal{B}$ corresponds to the dual operator of $\mathcal{L}_1$ with respect to $f$ and on the suited weighted $L_2$ space. The formula requires that $h(\mathbf{x},\mathbf{d})$ never vanishes. This condition is given by eq. (\ref{eq:condition_immersion}) in Lemma \ref{lemma:immersion_sphere}. This equation delivers a condition on the support $\Omega$ regarding the detector set $\mathbb{D}$. Allowed $\Omega$ for different position of sources are depicted in Figure \ref{fig:immersion_source_sphere} within the sphere $\mathbb{D}$. The choice for the architecture will constitute a compromise on the type of applications, typically the size of the objects and the size of the designed scanner. In the simulations, we consider the second case: the source is slightly shifted from the pole which allows a sufficiently large $\Omega$  in the center of the sphere. The fourth case, source at the center, maximizes the space for the object but since the source must be exterior to the object, it limitates solid objects to only a quarter of the sphere.

Equation (\ref{eq:FBP_CSI}) behaves as a FIO of order 1. Since $g_2$ is structurally smoother than $g_1$, it is expected that $\tilde{f}$ will carry the singularity of order 0 of $f$ up to singularities of lesser orders arising from $g_2$ and the rest of measurement $\eta$. This intuition will be the core of future researches to determine properly how are encoded and reconstructed the singularities of $f$.

Following on from \cite{RH_2018}, eq. (\ref{eq:FBP_CSI}) requires the weight $w_1 = \mathcal{W}_1(n_e)$ (and thus the electron density $n_e$) to be known to recover accurately the contours of $f$. To perform the reconstruction, we consider an approximation of $n_e$ as prior information (obtained from previous experiments for instance). The functions $f, n_e$ and $n_e^{prior}$ used for our simulations are depicted in Figure \ref{fig:inputs}. The analysis of the reconstruction operator with inexact weights will be performed in future research. 

\begin{figure}[!t]\centering
\includegraphics[width=0.9\linewidth]{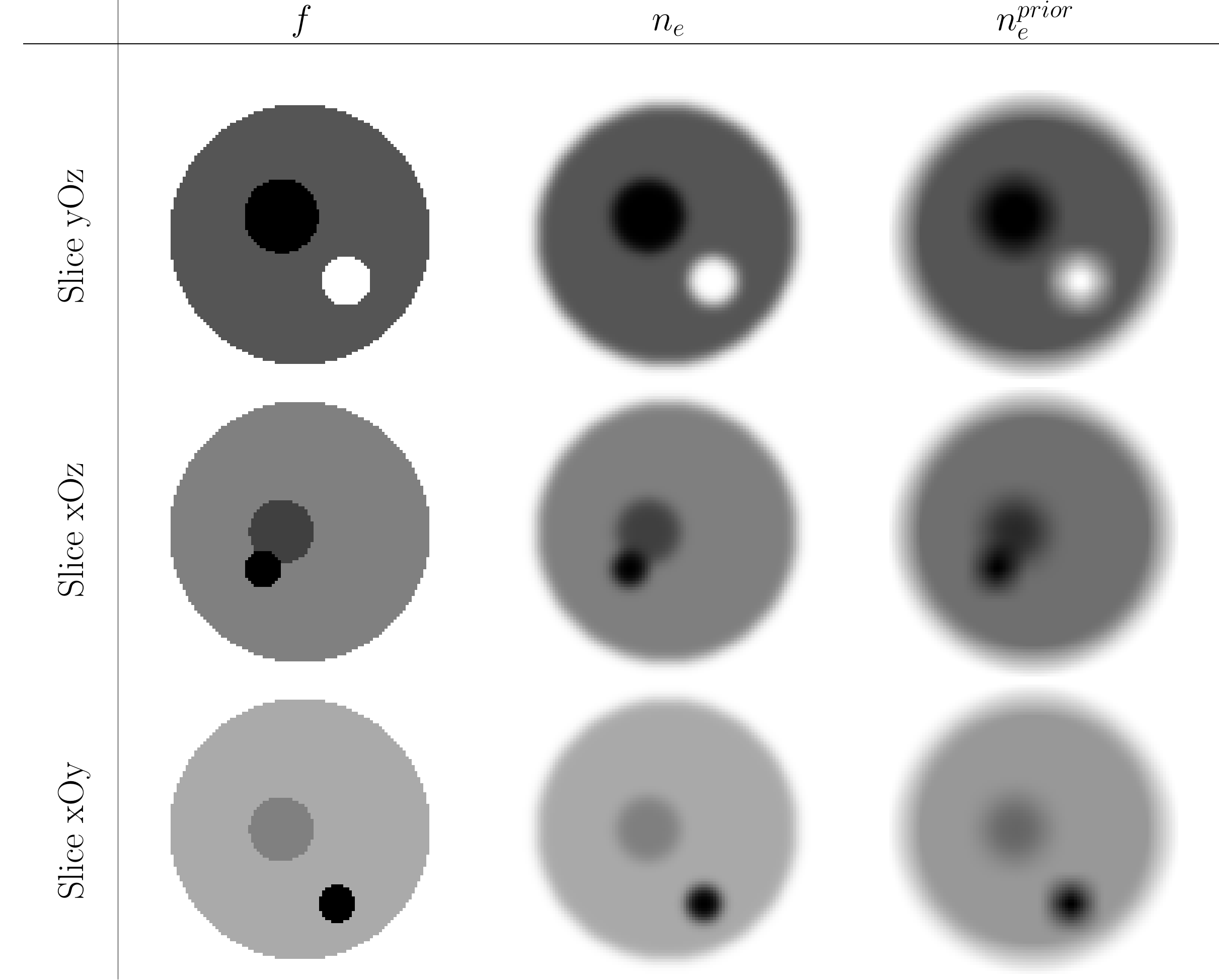}
\caption{Central slices of $f, n_e$ and $n_e^{prior}$.
\label{fig:inputs}}
\end{figure}

We now provide simulation results for the configuration given in Figure \ref{fig:config_CSI} in which the source has been slightly shifted from the pole as explained above. In the following, we assume that the detector has a continuous energy resolution. In practice, a fixed energy resolution will alter the accuracy of our integral representation for the back-scattered photons ($\omega>\pi/2$) leading to inconsistent data and thus limited data artifacts. This aspect constitutes the next step of our future research.

The parameters of the scanner are identical with the ones in Subsection \ref{ssec:MCPSF}. The toy object is a composition of spheres with electron densities $\{0,1,1.5,2,3\}\times n_w$ compactly supported in a cube of dimension 10x10x10cm$^3$. We implement the measurement under analytical form and Monte-Carlo simulation as well as the reconstruction technique we propose, see eqs. (\ref{eq:FBP_CSI}) and (\ref{eq:nabla_FBP_CSI}). The spectrum and its decomposition are depicted in Figure \ref{fig:dataMC_1_2_vs_ana}. In order to \textit{scan} uniformly the objects by the lemon and apple tori, and since we assume the detectors to measure continuously the spectrum, one can consider the mapping $\omega \to \tau$ defined by
$$
\tau = \Vert \mathbf{d} - \mathbf{s} \Vert \tan\left( \frac{\omega}{2} \right).
$$
This mapping corresponds to sample linearly the distance between the tori and the middle point $\mathbf{s}+\frac12(\mathbf{d} - \mathbf{s})$ which describes the \textit{size} of the apple or lemon tori. Without this property of the detectors, the back-scattered radiation ($\omega > \pi/2$) provides less reliable information and thus limited data issues.

\begin{figure}[!t]\centering
\includegraphics[width=0.8\linewidth]{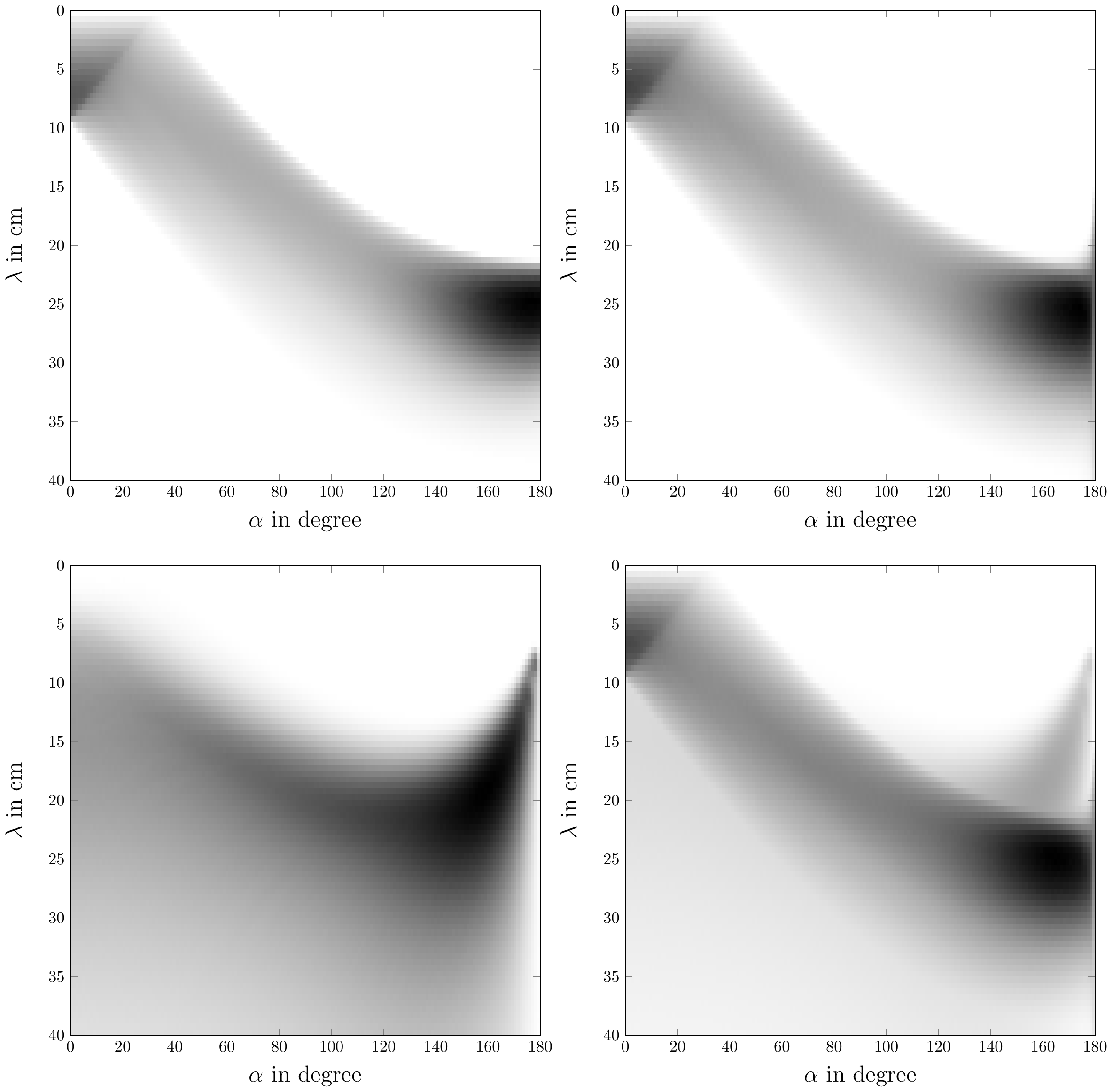}
\caption{Different parts of the spectrum from the electron density $n_e$ depicted in Figure \ref{fig:Rec_ne}. top left: implementation of the integral representation of $g_1$; top right: corresponding Monte-Carlo data $g_1^{MC}$; bottom left: $g_2^{MC}$ obtained by Monte-Carlo simulation; bottom right: $g_1^{MC}+g_2^{MC}$  \label{fig:dataMC_1_2_vs_ana}}
\end{figure}

In order to model the statistical nature of the emission and scattering of photons, we consider a Poisson process. This noise is characteristic for the emission of an ionising source and of scattering. It is characterized by the Poisson distribution
$$
Pr(x = k) = \frac{\varrho^k}{k!} e^{-\varrho}
$$
with $\varrho$ the average number of events per interval. Thus, when we consider noisy data, $\bar{g}$ will be replaced by $g^\epsilon$ where the values  $g_{jk}^\epsilon$ are drawn following the Poisson distribution with $\varrho = g_{jk}$. For the simulations, we have considered 0.5 \% noise obtained taking $I_0$ to $10^{11}$.

\begin{figure}[!t]\centering
\includegraphics[width=\linewidth]{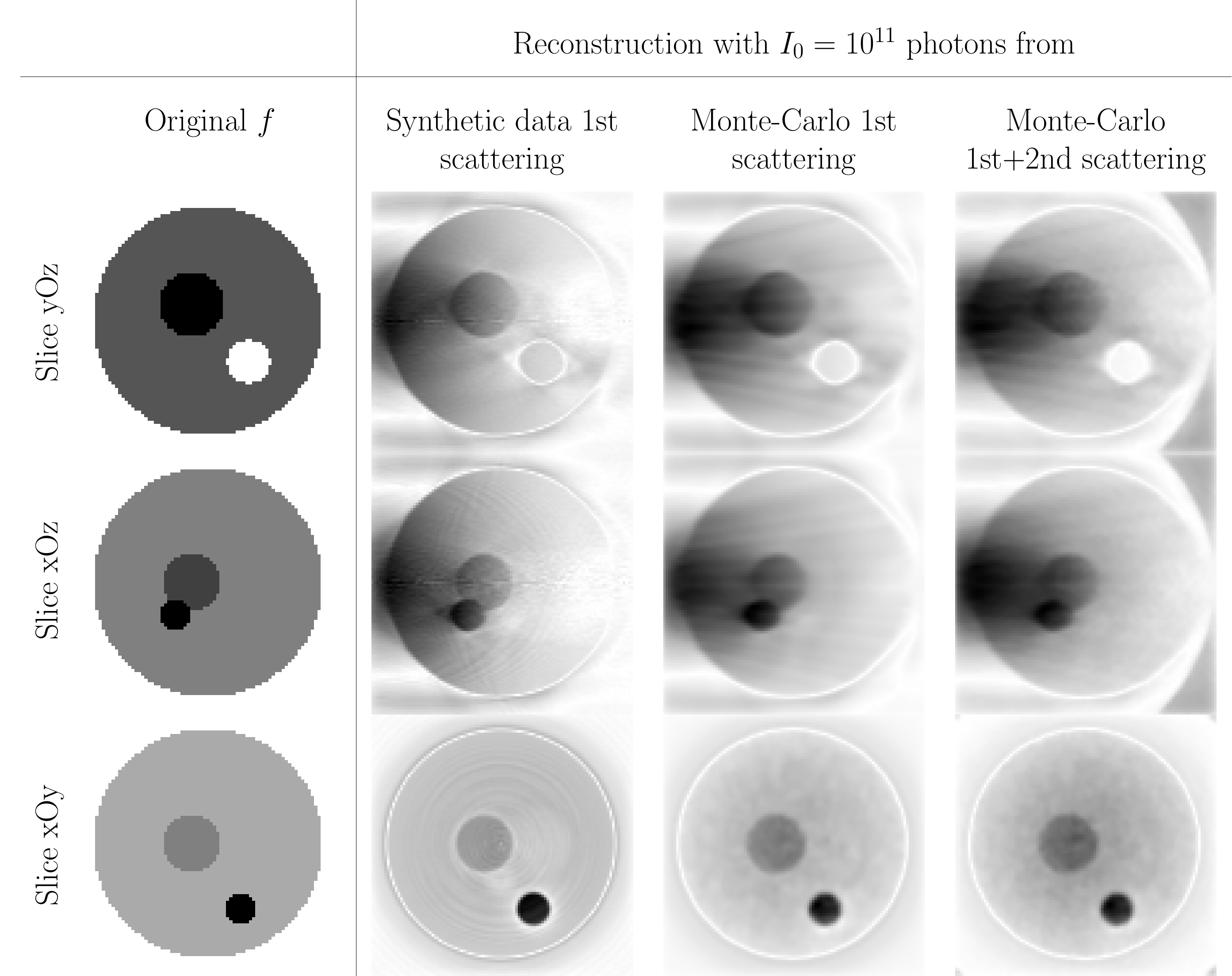}
\caption{first column: original electron density; Columns 2 to 4: Reconstruction from $g_1$, $g_1^{MC}$ and $g_1^{MC} + g_2^{MC}$ respectively using the reconstruction formula given in eq. (\ref{eq:FBP_CSI}). The level of noise in the Monte-Carlo data was set to 0.5\%. \label{fig:Rec_ne}}
\end{figure}

\begin{figure}[!t]\centering
\includegraphics[width=\linewidth]{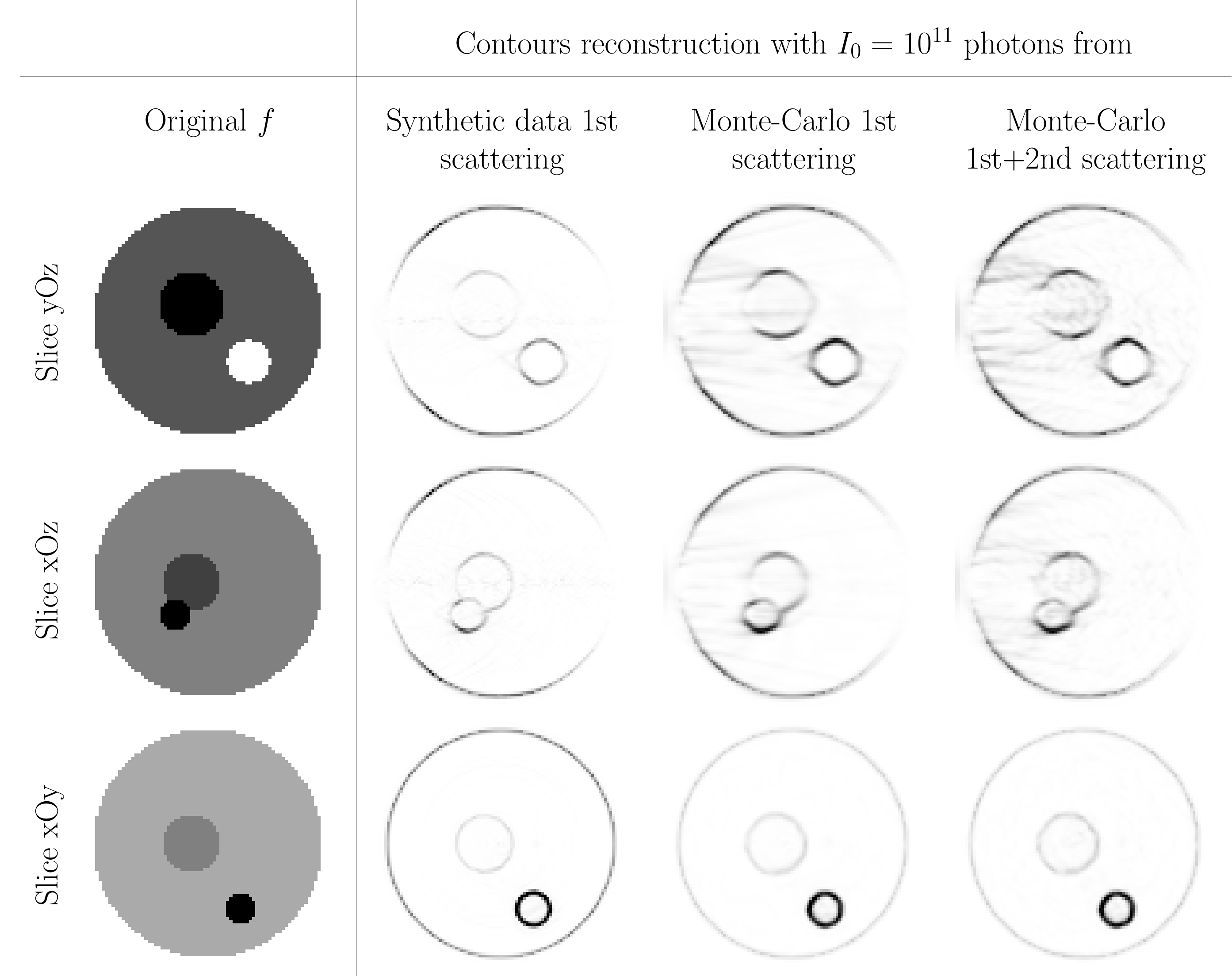}
\caption{Corresponding gradients to the reconstructions in Figure \ref{fig:Rec_ne} using eq. (\ref{eq:nabla_FBP_CSI}). \label{fig:Rec_contours_ne}}
\end{figure}

The reconstructions are displayed in Figure \ref{fig:Rec_ne}. As anticipated, the application of the filtered back-projection type algorithm does not provide a satisfactory reconstruction of the object. This is essentially due to the physical factors which alter substantially the integral kernel. They produce $C^\infty$-smooth but strong artifacts. By analogy with the attenuated Radon transform \cite{RG_IP_Sobolev_15}, the ill-conditioning of the reconstruction problem should increase exponentially with the intensity of the electron density which is observed in practice. The values of the electron density considered here as well as the variation of the physical factors are thus substantially amplified by the reconstruction strategy. 

However, the extraction of the contours, Figure \ref{fig:Rec_contours_ne}, enables a better visualization of the features of the object. Here, we simply take the gradient of the reconstructions from Figure \ref{fig:Rec_ne} but more sophisticated techniques could be applied. The second column displays the contours obtained from the analytical data for $g_1$. In this case, the contours are well-recovered. However, we can observe some artifacts arising from the variations of the physical factors. These artifacts turn out more prominent for Monte-Carlo data in the third and fourth columns. These can be explained by the smoothing effect of the acquisition, giving that the modeling itself is in reality not the torus but a strip around the torus giving smoother behaviour of the measurement than for the analytical modeling. Thence, the contours are smoothed making the artifacts more visible. Furthermore, as developed in this paper, the contours are encoded and preserved by the first-order scattered radiation and can be recovered even when considering the second-order scattering (fourth column). The key point here is that the second scattering is smoother than the first scattering. Consequently, the second derivative $\partial_p^2$ in the reconstruction scheme, highlights the variations of $g_1$ over $g_2$ and leads to a reconstruction almost not hindered by the second scattering. We expect the same behaviour for higher-ordered scattered radiation in the spectrum.

Due to the computation time of the data and of the Monte-Carlo simulations, we restricted the sampling of our toy object to 100x100x100 voxels which is small for recovering contours. We expect sharper reconstructions of the edges and less prominent artifacts for higher resolution of the data.

\section{Conclusion}

Restricting our study of the measured spectrum to the first- and second-order, we expressed the measurement of modalities on 3D Compton scattering imaging in terms of integral operators and approximated it under certain assumptions as Fourier integral operators. The study of these FIO delivered smoothness properties showing that the second-order scattered radiation provides smoother than the first-order scattered radiation. In consequence, the edges of the electron density are encoded essentially in the first-order part of the measured spectrum and can be recovered analytically using filtered backprojection techniques.

\section*{Appendix}

\subsection*{Proof of Theorem \ref{theo:model_scat2}}
The structure follows on from the physics of Compton scattering. Akin to the first scattering, the variation of photon scattered twice can be expressed as
$$
\mathrm{d}^2 g_2 = I_0 \left(\frac12 r_e\right)^4 P(\omega_1) P(\omega_2) A_{E_0}(\mathbf{s} , \mathbf{x}) A_{E_{\omega_1}}(\mathbf{x}, \mathbf{y}) A_{E_{\omega_2}}(\mathbf{y}, \mathbf{d}) n_e(\mathbf{x}) n_e(\mathbf{y}) \mathrm{d}\mathbf{x} \mathrm{d}\mathbf{y}.
$$
Therefore, ignoring some physical constants and based on our development above, one can integrate to get the theoretical number of photons detected at $\mathbf{d}$ with energy $E$ after two scattering events, 
$$
g_2(E,\mathbf{d},\mathbf{s}) = \underset{\Omega}{\int} \underset{\mathfrak{T}(\omega_2(\omega_1),\mathbf{x},\mathbf{d}) \cap \mathfrak{C}(\omega_1,\mathbf{x})}{\int} \hspace{-3em} A_{E_0}(\mathbf{s} , \mathbf{x}) A_{E_{\omega_1}}(\mathbf{x}, \mathbf{y}) A_{E_{\omega_2}}(\mathbf{y}, \mathbf{d}) n_e(\mathbf{x}) n_e(\mathbf{y}) \mathrm{d}\mathbf{y} \mathrm{d}\mathbf{x}.
$$
This intersection is characterized by eqs. (\ref{eq:coordinates_intersection}) and (\ref{eq:radius_intersection}). We still have to compute the differential form along the intersection and thus
$$
\Vert \partial_{\omega_1} \mathbf{y}_\cap \wedge \partial_{\varphi} \mathbf{y}_\cap \Vert
$$
for $\mathbf{x},\mathbf{d},\mathbf{s}$ given. First, one can compute
$$
\partial_{\varphi} \mathbf{y}_\cap  = R_2 R_1 
\left(
r_\cap  
 \left(
\begin{array}{c}
-\sin \omega_1 \sin \varphi \\ 
\sin \omega_1 \cos \varphi \\ 
0
\end{array}\right)
+
\partial_\varphi r_\cap  
 \left(
\begin{array}{c}
\sin \omega_1 \cos \varphi \\ 
\sin \omega_1 \sin \varphi \\ 
\cos \omega_1
\end{array}\right)
\right)
$$
and 
$$
\partial_{\omega_1} \mathbf{y}_\cap  = R_2 R_1 
\left(
r_\cap  
 \left(
\begin{array}{c}
\cos \omega_1 \cos \varphi \\ 
\cos \omega_1 \sin \varphi \\ 
-\sin \omega_1
\end{array}\right)
+
\partial_{\omega_1} r_\cap  
 \left(
\begin{array}{c}
\sin \omega_1 \cos \varphi \\ 
\sin \omega_1 \sin \varphi \\ 
\cos \omega_1
\end{array}\right)
\right).
$$
This leads to
\begin{eqnarray*}
\partial_\varphi \mathbf{y}_\cap \wedge \partial_{\omega_1} \mathbf{y}_\cap 
&=&  R_2 R_1 
\left(
r_\cap^2 
 \left(
\begin{array}{c}
\sin^2 \omega_1 \cos \varphi \\ 
\sin^2 \omega_1 \sin \varphi \\ 
\cos \omega_1 \sin \omega_1
\end{array}\right)
+ r_\cap \partial_\varphi r_\cap 
 \left(
\begin{array}{c}
\sin \varphi \\ -\cos\varphi \\ 0
\end{array}\right) \right. \\
&+& \left.  r_\cap \partial_{\omega_1} r_\cap 
 \left(
\begin{array}{c}
-\cos \omega_1 \sin \omega_1 \cos \varphi\\
-\cos \omega_1 \sin \omega_1 \sin \varphi \\ 
\sin^2 \omega_1
\end{array}
\right)
\right).
\end{eqnarray*}
Since the rotation matrices do not change the norm, they can be ignored in the computation of the norm. After some computations, one gets
$$
\Vert \partial_{\omega_1} \mathbf{y}_\cap \wedge \partial_{\varphi} \mathbf{y}_\cap \Vert = r_\cap  \sqrt{\sin^2\omega_1 (r_\cap^2 +  \left( \partial_{\omega_1} r_\cap  \right)^2) + \left( \partial_{\varphi} r_\cap )  \right)^2}
$$
where
\begin{eqnarray*}
\partial_{\omega_1} r_\cap &=& (z_{\cap})_{\omega_1} \left(1+  \cot \omega_2 \frac{z_\cap}{\sqrt{1-z_\cap^2}} \right) - \frac{\sin \omega_1}{\sin^3 \omega_2} \sqrt{1-z_\cap^2} \\
\partial_{\varphi} r_\cap &=& (z_{\cap})_{\varphi} \left(1 + \cot \omega_2  \frac{z_\cap}{\sqrt{1-z_\cap^2}} \right)
\end{eqnarray*}
in which 
\begin{eqnarray*}
(z_{\cap})_{\omega_1} &=& R_1(3,1) \cos \omega_1 \cos \varphi
+ R_1(3,2) \cos \omega_1 \sin \varphi - R_1(3,3) \sin  \omega_1\\
(z_{\cap})_{\varphi} &=&  R_1(3,2) \sin \omega_1 \cos \varphi - R_1(3,1) \sin \omega_1 \sin \varphi .
\end{eqnarray*}
This ends the proof.

\subsection*{Proof of Lemma \ref{theo:1to1}}

First, we discard the degenerate case of the spindle torus which occurs when $\omega = 0$ (there is no scattering event, only primary radiation) and corresponds to the line $\{\mathbf{s}+t(\mathbf{d}-\mathbf{s}), t\in \mathbb{R}\}$. 

The parameter representation of the spindle torus is given by
$$
\mathfrak{T}(\omega,\mathbf{d},\mathbf{s}) = 
\left\{
\mathbf{x} = \mathbf{s} + \Vert \mathbf{d}-\mathbf{s}\Vert \frac{\sin(\omega - \alpha)}{\sin \omega} R_2
\left(
\hspace{-4pt}
\begin{array}{c}
\sin \alpha \cos \beta \\ 
\sin \alpha \sin \beta \\ 
\cos \alpha 
\end{array}
\hspace{-4pt}
\right) 
: \alpha \in [0,\omega], \beta \in [0,2\pi)
\right\}
$$
with $p = \cot \omega$ and $R_2$ the rotation matrix which maps 
$
e_z = (0,0,1)^T
$
into
$
\frac{\mathbf{d} - \mathbf{s}}{\Vert \mathbf{d} - \mathbf{s} \Vert}
$. For $\mathbf{s}$ and $\mathbf{d}$ fixed, it is clear that
$$
 R_2
\left(
\hspace{-4pt}
\begin{array}{c}
\sin \alpha \cos \beta \\ 
\sin \alpha \sin \beta \\ 
\cos \alpha 
\end{array}
\hspace{-4pt}
\right) 
$$
is one-to-one with the unit sphere. This is why each $\mathbf{x} \in \mathbb{R}^3$ corresponds to only one pair $(\alpha,\beta)$. Now, since $\frac{\sin(\omega - \alpha)}{\sin \omega}$ is defined for $\omega > \alpha$, the norm of the vector $\mathbf{x} - \mathbf{s}$ is a bijective function on $\omega$. 

Therefore, the spindle torus parametrized by $(\omega,\alpha,\beta)$ is one-to-one with $\mathbb{R}^3\setminus\{\mathbf{s}+t(\mathbf{d}-\mathbf{s}), t\in \mathbb{R}\}$.

\subsection*{Proof of Theorem \ref{theo:continuity_T}}
From \cite{RH_2018}, we know that given a spherical parametrization $\mathbf{x}(p,\eta_1,\eta_2)$ of the torus (lemon or apple part) for $p,\mathbf{d},\mathbf{s}$ fixed
$$ 
\mathcal{T} (n_e) (p,\mathbf{d},\mathbf{s}) = \int_0^{2\pi} \int_0^\omega w_1(\mathbf{x}(p,\eta_1,\eta_2),p,\mathbf{d},\mathbf{s}) n_e(\mathbf{x}(p,\eta_1,\eta_2)) J(p,\eta_1,\eta_2) \mathrm{d}\eta_1 \mathrm{d}\eta_2
$$
with $p = \cot \omega$ and $J$ the appropriate Jacobian given in \cite{RH_2018}. Using the Cauchy-Schwarz inequality, one gets
\begin{eqnarray*}
&\vert \mathcal{T}(f_n) (p,\mathbf{d},\mathbf{s}) - \mathcal{L}_1 (f,f_n)  (p,\mathbf{d},\mathbf{s}) \vert^2 \\
&\leq
\int_0^{2\pi} \int_0^\omega (w_1(\mathbf{x}(p,\eta_1,\eta_2),p,\mathbf{d},\mathbf{s}))^2 \chi_\Omega(\eta_1,\eta_2) J(p,\eta_1,\eta_2) \mathrm{d}\eta_1 \mathrm{d}\eta_2 \\
\cdot
&\int_0^{2\pi} \int_0^\omega |f_n(\mathbf{x}(p,\eta_1,\eta_2))-f(\mathbf{x}(p,\eta_1,\eta_2)) |^2 J(p,\eta_1,\eta_2) \mathrm{d}\eta_1 \mathrm{d}\eta_2 
\end{eqnarray*}
in which $\chi_\Omega$ stands for a smooth cut-off well-defined since $f,n_e$ are compactly supported in $\Omega$ an open subset of $\mathbb{R}^3$. Taking now the $L_2$ norm yields
$$
\Vert  \mathcal{T}(f_n)- \mathcal{L}_1 (f,f_n) \Vert_{L_2} \leq c \int_\mathbb{R} \int_0^{2\pi} \int_0^\omega  |f_n(\mathbf{x}(p,\eta_1,\eta_2))-f(\mathbf{x}(p,\eta_1,\eta_2)) |^2 J(p,\eta_1,\eta_2) \mathrm{d}\eta_1 \mathrm{d}\eta_2 \mathrm{d}p
$$
with 
$$
c = \int_\mathbb{R} \int_0^{2\pi} \int_0^\omega (w_1(\mathbf{x}(p,\eta_1,\eta_2),p,\mathbf{d},\mathbf{s}))^2 \chi_\Omega(\eta_1,\eta_2) J(p,\eta_1,\eta_2) \mathrm{d}\eta_1 \mathrm{d}\eta_2 \mathrm{d}p
$$
which is well-defined as the integrand is a compactly supported smooth
function. Due to Lemma \ref{theo:1to1} and since the discarded line passing through $\mathbf{s}$ and $\mathbf{d}$ is negligible regarding the Lebesgue measure, one can finally apply the mapping $(p,\eta_1,\eta_2) \mapsto \mathbf{y}$ and gets
$$
\Vert  \mathcal{T}(f_n)- \mathcal{L}_1 (f,f_n)  \Vert_{L_2}^2 \leq c \int_{\Omega} |f(\mathbf{y})-f_n(\mathbf{y})|^2 \mathrm{d} \mathbf{y} = c \ \Vert f - f_n \Vert_{L_2}^2.
$$
Taking the limit $n \to \infty$ ends the proof.

\subsection*{Proof of Lemma \ref{lemma:immersion_sphere}}

We start from
$$
\phi(\mathbf{x},\mathbf{d},\mathbf{s}) = \frac{\kappa(\mathbf{x}) - \rho(\mathbf{x}) }{\sqrt{1- \kappa^2(\mathbf{x})}} \quad \mbox{with} \quad \kappa(\mathbf{x}) = \frac{\mathbf{x}-\mathbf{s}}{\Vert \mathbf{x}-\mathbf{s}\Vert} \cdot \frac{\mathbf{d}-\mathbf{s}}{\Vert \mathbf{d}-\mathbf{s}\Vert}, \ \rho(\mathbf{x}) = \frac{\Vert \mathbf{x}-\mathbf{s}\Vert}{ \Vert \mathbf{d}-\mathbf{s}\Vert}
$$ 
and
\begin{eqnarray*}
\nabla_{\mathbf{x}}\phi(\mathbf{x},\mathbf{d},\mathbf{s})  
&=& \frac{(1- \rho(\mathbf{x}) \kappa(\mathbf{x}) )}{\Vert \mathbf{x}-\mathbf{s}\Vert (1-\kappa^2(\mathbf{x}))^{3/2}} \frac{\mathbf{d}-\mathbf{s}}{\Vert \mathbf{d}-\mathbf{s} \Vert} \nonumber\\
&-& \frac{1}{\Vert \mathbf{x}-\mathbf{s}\Vert \sqrt{1-\kappa^2(\mathbf{x})}} 
\left( \rho(\mathbf{x}) + \kappa(\mathbf{x}) \frac{1-\rho(\mathbf{x})\kappa(\mathbf{x})}{1-\kappa^2(\mathbf{x})}
\right) \frac{\mathbf{x}-\mathbf{s}}{\Vert\mathbf{x}-\mathbf{s}\Vert}.
\end{eqnarray*}
In order to use the invariances of $\kappa(\mathbf{x})$, we consider a rotation $R_\mathbf{x}$ which maps $\frac{\mathbf{x}-\mathbf{s}}{\Vert\mathbf{x}-\mathbf{s}\Vert}$ to $(0,0,1)^T$. After this rotation, we get for $\mathbf{x}$ fixed
$$
\mathbf{x}-\mathbf{s} = 
\left(
\begin{array}{c}0\\0\\r\end{array}
\right)
\quad \text{and} \quad
\mathbf{d}-\mathbf{s} 
= t(\theta(\bar{\alpha},\bar{\beta})) 
\left(\begin{array}{c}
\sin \bar{\alpha} \cos \bar{\beta}\\ 
\sin \bar{\alpha} \sin \bar{\beta}\\ 
\cos \bar{\alpha}
\end{array} \right) 
$$
in which $\bar{\alpha}$ and $\bar{\beta}$ stands for the elevation and azimuth angles of $\mathbf{d}-\mathbf{s}$ in the new system of spherical coordinates oriented by $\mathbf{x}$. Since $\mathbf{x}$ is fixed, we omitted their dependency with respect to $\mathbf{x}$. $\theta(\bar{\alpha},\bar{\beta})$ expresses the elevation angle of $\mathbf{d}-\mathbf{s}$ regarding $\mathbf{x}-\mathbf{s}$ in the new system of coordinates oriented by $\mathbf{x}-\mathbf{s}$ and parametrized by $\bar{\alpha}$ and $\bar{\beta}$.\\[1em]
For $\mathbf{x}$ given and after the rotation $R_\mathbf{x}$, the phase and its gradient simplify to 
$$
\bar{\phi}(\bar{\alpha},\bar{\beta})  = \frac{\cos(\bar{\alpha}) - r/t(\theta(\bar{\alpha},\bar{\beta}))}{\sin \bar{\alpha}}
\quad \text{and} \quad
\bar{\phi}_\mathbf{x} (\bar{\alpha},\bar{\beta}) 
= \frac{1}{t \sin^2 \bar{\alpha}}  \left(\begin{array}{c}
\left(  \cos \bar{\alpha} - \frac{t}{r} \right)  \cos \bar{\beta}\\ 
\left(  \cos \bar{\alpha} - \frac{t}{r} \right)  \sin \bar{\beta}\\ 
\sin \bar{\alpha}
\end{array} \right) .
$$ 
We must now compute the derivative of $\bar{\phi}_\mathbf{x}(\bar{\alpha},\bar{\beta}) $ regarding $\bar{\alpha}$ and $\bar{\beta}$. After some computations, one gets 
$$
\partial_{\bar{\alpha}} \bar{\phi}_\mathbf{x}(\bar{\alpha},\bar{\beta})
= \frac{1}{t \sin^3 \bar{\alpha}} \left(\begin{array}{c}
\left[- \frac{2 t\cos \bar{\alpha}}{r} + 1 + \cos^2 \bar{\alpha} + \frac{t_{\bar{\alpha}}}{t} \cos\bar{\alpha} \sin\bar{\alpha} \right] \cos \bar{\beta}\\ 
\left[- \frac{2 t\cos \bar{\alpha}}{r} + 1 + \cos^2 \bar{\alpha} + \frac{t_{\bar{\alpha}}}{t} \cos\bar{\alpha} \sin\bar{\alpha} \right] \sin \bar{\beta}\\ 
\cos \bar{\alpha} \sin \bar{\alpha} + \frac{t_{\bar{\alpha}}}{t} \sin^2 \bar{\alpha}
\end{array} \right)
$$
and 
$$
\partial_{\bar{\beta}} \bar{\phi}_\mathbf{x}(\bar{\alpha},\bar{\beta})=  \frac{1}{t \sin^2 \bar{\alpha}} \left[ \left( \frac{t}{r} - \cos \bar{\alpha} \right) \left(\begin{array}{c}
- \sin \bar{\beta}\\ 
 \cos \bar{\beta}\\ 
0
\end{array} \right)
+ \frac{t_{\bar{\beta}}}{t}
\left(\begin{array}{c}
\cos\bar{\alpha} \cos \bar{\beta}\\ 
\cos\bar{\alpha} \sin \bar{\beta}\\ 
\sin \bar{\alpha}
\end{array} \right)
\right].
$$
Finally, defining the determinant of the Jacobian matrix after rotation $R_\mathbf{x}$ and for $\mathbf{x}$ given by 
$$
\bar{h}(\bar{\alpha},\bar{\beta}):=\mathrm{det}(\bar{\phi}_\mathbf{x}, \partial_{\bar{\alpha}} \bar{\phi}_\mathbf{x}, \partial_{\bar{\beta}} \bar{\phi}_\mathbf{x})(\bar{\alpha},\bar{\beta}), \quad \text{for }  (\bar{\alpha},\bar{\beta}) \in [0,\pi] \times [0,2\pi],
$$ 
it yields 
$$
\bar{h}(\bar{\alpha},\bar{\beta}) =  
\frac{r \cos\bar{\alpha} - t}{t^3 \sin^6 \bar{\alpha}}(r+t_{\bar{\alpha}} \sin \bar{\alpha} - t\cos\bar{\alpha}).
$$ 
In the computations, the part with $t_{\bar{\beta}}$ cancels out. Since the rotation $R_\mathbf{x}$ does not change the linear independence of $\{\nabla_\mathbf{x} \phi, \partial_\alpha \nabla_\mathbf{x} \phi, \partial_\beta \nabla_\mathbf{x} \phi\}$, $h(\mathbf{x},\mathbf{d})$ vanishes thus only if 
\begin{enumerate}
\item $r\cos\bar{\alpha} =  t$ which corresponds to the condition in eq. (\ref{eq:condition_immersion}), or
\item $t_{\bar{\alpha}} \sin \bar{\alpha} = t \cos \bar{\alpha} - r$ which  is our condition on $\mathbb{D}$. Indeed $\bar{\alpha} = \arccos(\kappa(\mathbf{x}))$ thus $t_{\bar{\alpha}}$ becomes after the rotation $R_\mathbf{x}^{-1}$ becomes 
$$
- \frac{\partial_\alpha \kappa(\mathbf{x})}{\sqrt{1-\kappa^2(\mathbf{x})}} \partial_\alpha t(\alpha).
$$
\end{enumerate}

\bibliographystyle{siamplain}
\bibliography{references}

\begin{thebibliography}{10}

\bibitem{ABE_11}
{\sc O.~O. Adejumo, F.~A. Balogun, and G.~G.~O. Egbedokun}, {\em Developing a
  compton scattering tomography system for soil studies: Theory}, Journal of
  Sustainable Development and Environmental Protection, 1 (2011), pp.~73--81.

\bibitem{AM_76}
{\sc R.~Alvarez and A.~Macovski}, {\em {Energy-selective reconstructions in
  x-ray computerized tomography}}, Phys Med Biol., 21 (1976), pp.~pp. 733--744.

\bibitem{AHD_90}
{\sc S.~Anghaie, L.~L. Humphries, and N.~J. Diaz}, {\em {Material
  characterization and flaw detection, sizing, and location by the differential
  gamma scattering spectroscopy technique. Part1: Development of theoretical
  basis}}, Nuclear Technology, 91 (1990), pp.~361--375.

\bibitem{Arendtsz}
{\sc N.~V. Arendtsz and E.~M.~A. Hussein}, {\em {Energy-spectral Compton
  scatter Imaging - Part 1: theory and mathematics}}, IEEE Transactions on
  Nuclear Sciences, 42 (1995), pp.~2155--2165.

\bibitem{BC_03}
{\sc F.~A. Balogun and P.~E. Cruvinel}, {\em {Compton scattering tomography in
  soil compaction study}}, Nuclear Instruments and Methods in Physics Research
  A, 505 (2003), pp.~502--507.

\bibitem{BCGLR_02}
{\sc A.~Brunetti, R.~Cesareo, B.~Golosio, P.~Luciano, and A.~Ruggero}, {\em
  {Cork quality estimation by using Compton tomography}}, Nuclear instruments
  and methods in Physics research B, 196 (2002), pp.~161--168.

\bibitem{Cesareo_02}
{\sc R.~Cesareo, C.~C. Borlino, A.~Brunetti, B.~Golosio, and A.~Castellano},
  {\em {A simple scanner for Compton tomography}}, Nuclear Instruments and
  Methods in Physics Research A, 487 (2002), pp.~188--192.

\bibitem{CV_73}
{\sc R.~L. Clarke and G.~V. Dyk}, {\em A new method for measurement of bone
  mineral content using both transmitted and scattered beams of gamma-rays},
  Phys. Med. Biol., 18 (1973), pp.~532--539.

\bibitem{Compton_23}
{\sc A.~H. Compton}, {\em A quantum theory of the scattering of x-rays by light
  elements}, Phys. Rev., 21 (1923), pp.~483--502.

\bibitem{Driol}
{\sc C.~Driol}, {\em Imagerie par rayonnement gamma diffus\'e \`a haute
  sensibilit\'e}, PhD thesis, Univ. of Cergy-Pontoise, 2008.

\bibitem{Shefer2013}
{\sc {E. Shefer \textit{et al}}}, {\em {State of the Art of CT Detectors and
  Sources: A Literature Review}}, Current Radiology Reports, 1 (2013), pp.~pp.
  76--91.

\bibitem{EMBR_98}
{\sc B.~L. Evans, J.~B. Martin, L.~W. Burggraf, and M.~C. Roggemann}, {\em
  {Nondestructive inspection using Compton scatter tomography}}, IEEE
  Transactions on Nuclear Science, 45 (1998), pp.~950--956.

\bibitem{FC_71}
{\sc F.~T. Farmer and M.~P. Collins}, {\em {A new approach to the determination
  of anatomical cross-sections of the body by Compton scattering of
  gamma-rays}}, Phys. Med. Biol., 16 (1971), pp.~577--586.

\bibitem{Fredenberg_2018}
{\sc E.~Fredenberg}, {\em {Spectral and dual-energy X-ray imaging for medical
  applications. Nuclear Instruments and Methods in Physics Research Section A:
  Accelerators, Spectrometers, Detectors and Associated Equipment}}, 878
  (2018), pp.~pp 74--87.

\bibitem{GG_2017}
{\sc H.~Goo and J.~Goo}, {\em {Dual-Energy CT: New Horizon in Medical
  Imaging}}, Korean J Radiol., 18 (2017), pp.~pp. 555--569.

\bibitem{GKAD_05}
{\sc V.~A. Gorshkov, M.~Kroening, Y.~V. Anosov, and O.~Dorjgochoo}, {\em {X-Ray
  scattering tomography}}, Nondestructive Testing and Evaluation, 20 (2005),
  pp.~147--157.

\bibitem{Guzzardi}
{\sc R.~Guzzardi and G.~Licitra}, {\em {A critical review of Compton imaging}},
  CRC Critical Reviews in Biomedical Imaging, 15 (1988), pp.~237--268.

\bibitem{HH_2010}
{\sc G.~Harding and E.~Harding}, {\em {Compton scatter imaging: a tool for
  historical exploitation}}, Appl Radiat Isot., 68 (2010), pp.~993--1005.

\bibitem{Hormander}
{\sc L.~H{\"o}rmander}, {\em {Fourier Integral Operators, I}}, vol.~127, Acta
  Math., 1971.

\bibitem{H_73}
{\sc G.~Hounsfield}, {\em Computerized transverse axial scanning (tomography).
  i. description of system}, Br J Radiol., 46 (1973), pp.~pp. 1016--1022.

\bibitem{KN}
{\sc O.~Klein and Y.~Nishina}, {\em {\"U}ber die {S}treuung von {S}trahlung
  durch freie {E}lektronen nach der neuen relativistischen {Q}uantendynamik von
  {D}irac}, Z. Phys., 52 (1929), pp.~853--869.

\bibitem{Quinto}
{\sc V.~Krishnan and E.~T. Quinto}, {\em Microlocal Analysis in Tomography},
  Handbook of Mathematical Methods in Imaging, Book editor: Otmar Scherzer,
  2015.

\bibitem{Kuchment}
{\sc P.~Kuchment, K.~Lancaster, and L.~Mogilevskaya}, {\em {On local
  tomography}}, Inverse Problems, 11 (1995), p.~571.

\bibitem{Lale_59}
{\sc P.~G. Lale}, {\em {The Examination of Internal Tissues, using Gamma-ray
  Scatter with a Possible Extension to Megavoltage Radiography}}, Physics in
  Medicine and Biology, 4 (1959), pp.~159--167.

\bibitem{Louis}
{\sc A.~K. Louis}, {\em Approximate inverse for linear and some nonlinear
  problems}, Inverse Problems, 12 (1996), pp.~175--190.

\bibitem{MLLF_2015}
{\sc C.~McCollough, S.~Leng, Y.~Lifeng, and J.~Fletcher}, {\em {Dual- and
  Multi-Energy CT: Principles, Technical Approaches, and Clinical
  Applications}}, Radiology, 276 (2015), pp.~pp. 637--653.

\bibitem{Hussein}
{\sc D.~A. Meneley, E.~M.~A. Hussein, and S.~Banerjee}, {\em {On the solution
  of the inverse problem of radiation scattering imaging }}, Nuclear Science
  and Engineering, 92 (1986), pp.~341--349.

\bibitem{Natterer}
{\sc F.~Natterer}, {\em The mathematics of computerized tomography}, 2001.

\bibitem{Nguyen2011}
{\sc M.~Nguyen, T.~Truong, M.~Morvidone, and H.~Zaidi}, {\em {Scattered
  radiation emission imaging: Principles and applications}}, International
  Journal of Biomedical Imaging (IJBI),  (2011), p.~15pp.

\bibitem{Norton_94}
{\sc S.~J. Norton}, {\em {Compton scattering tomography}}, Jour. Appl. Phys.,
  76 (1994), pp.~2007--2015.

\bibitem{Palamodov}
{\sc V.~Palamodov}, {\em A uniform reconstruction formula in integral
  geometry}, Inverse Problems, 28 (2012), p.~065014.

\bibitem{Nguyen2017}
{\sc P.~Prado, M.~Nguyen, L.~Dumas, and S.~Cohen}, {\em {Three-dimensional
  imaging of flat natural and cultural heritage objects by a Compton scattering
  modality}}, Journal of Electronic Imaging, 26 (2017), p.~011026.

\bibitem{PRLYM_2009}
{\sc A.~Primak, J.~R. Giraldo, X.~Liu, L.~Yu, and C.~McCollough}, {\em
  {Improved dual-energy material discrimination for dual-source CT by means of
  additional spectral filtration}}, Med Phys., 36 (2009), pp.~pp. 1359--1369.

\bibitem{Rigaud_SIIMS_17}
{\sc G.~Rigaud}, {\em {Compton Scattering Tomography: Feature Reconstruction
  and Rotation-Free Modality}}, SIAM J. Imaging Sci., 10 (2017),
  p.~2217–2249.

\bibitem{RH_2018}
{\sc G.~Rigaud and B.~Hahn}, {\em {3D Compton scattering imaging and contour
  reconstruction for a class of Radon transforms}}, Inverse Problems, 2018 (7),
  p.~075004.

\bibitem{RG_IP_Sobolev_15}
{\sc G.~Rigaud and A.~Lakhal}, {\em {Approximate inverse and Sobolev estimates
  for the attenuated Radon transform }}, Inverse Problems, 31 (2015),
  p.~105010.

\bibitem{Stonestrom}
{\sc J.~P. Stonestrom, R.~E. Alvarez, and A.~Macovski}, {\em {A framework for
  spectral artifact corrections in x-ray CT}}, IEEE Trans. Biomed. Eng., 28
  (1981), pp.~128--141.

\bibitem{TM_2015}
{\sc B.~Tracey and E.~Miller}, {\em {Stabilizing dual-energy X-ray computed
  tomography reconstructions using patch-based regularization}}, Inverse
  Problems, 31 (2015), p.~05004.

\bibitem{WS_17}
{\sc J.~{Webber} and S.~{Holman}}, {\em {Microlocal analysis of a spindle
  transform}}, arXiv e-prints,  (2017), arXiv:1706.03168, p.~arXiv:1706.03168,
  \url{https://arxiv.org/abs/1706.03168}.

\bibitem{WL_17}
{\sc J.~Webber and W.~Lionheart}, {\em {Three dimensional Compton scattering
  tomography}}, arXiv e-prints,  (2017), arXiv:1704.03378, p.~arXiv:1704.03378,
  \url{https://arxiv.org/abs/1704.03378}.

\bibitem{WQ_19}
{\sc J.~Webber and E.~Quinto}, {\em {Microlocal analysis of a Compton
  tomography problem}}, arXiv e-prints,  (2019), arXiv:1902.09623,
  p.~arXiv:1902.09623, \url{https://arxiv.org/abs/1902.09623}.

\bibitem{Wolfram}
{\sc E.~Weisstein}, {\em Spindle torus}, From MathWorld--A Wolfram Web
  Resource, \url{http://mathworld.wolfram.com/SpindleTorus.html}.

\end{thebibliography}

\end{document}